\documentclass[12pt]{amsart}
\usepackage[utf8]{inputenc}
\usepackage{hyperref}
\usepackage[final]{showkeys} 
\usepackage{amssymb}
\usepackage{amsmath}
 \usepackage{amsthm}
\usepackage{enumitem}
\usepackage{tikz -inet}
\usetikzlibrary{positioning,shapes.misc, patterns,arrows,decorations,decorations.pathreplacing, snakes}

\newtheorem{theorem}{Theorem}[section]
\newtheorem{hypothesis}[theorem]{Hypothesis}

\newtheorem{definition}[theorem]{Definition}
\newtheorem{proposition}[theorem]{Proposition}
\newtheorem{question}[theorem]{Question}
\newtheorem{lemma}[theorem]{Lemma}
\newtheorem{fact}[theorem]{Fact}
\newtheorem{remark}[theorem]{Remark}
\newtheorem{corollary}[theorem]{Corollary}
\newtheorem{claim}[theorem]{Claim}

\newcommand{\fct}[2]{{}^{#1}#2}

\newcommand{\K}{\mathcal{K}}
\newcommand{\Ksatp}[1]{\K^{#1\text{-sat}}}
\newcommand{\rest}{\upharpoonright}
\newcommand{\Union}{\bigcup}
\DeclareMathOperator{\tp}{ga-tp}
\DeclareMathOperator{\id}{id}
\DeclareMathOperator{\Aut}{Aut}
\DeclareMathOperator{\cf}{cf}
\newcommand{\seq}[1]{\langle #1 \rangle}
\newcommand{\gaS}{\operatorname{ga-S}}
\newcommand{\gS}{\gaS}
\newcommand{\gtp}{\tp}
\DeclareMathOperator{\LS}{LS}
\newcommand{\T}{\mathcal{T}}
\def\rg{\operatorname{rg}}
\newcommand{\C}{\mathfrak{C}}

\newcommand{\sea}{\mathfrak{C}}
\newcommand{\hanf}[1]{h (#1)}

\newcommand{\ba}{\bar{a}}
\newcommand{\bb}{\bar{b}}
\newcommand{\bc}{\bar{c}}

\newbox\noforkbox \newdimen\forklinewidth
\forklinewidth=0.3pt \setbox0\hbox{$\textstyle\smile$}
\setbox1\hbox to \wd0{\hfil\vrule width \forklinewidth depth-2pt
 height 10pt \hfil}
\wd1=0 cm \setbox\noforkbox\hbox{\lower 2pt\box1\lower
2pt\box0\relax}
\def\unionstick{\mathop{\copy\noforkbox}\limits}

\setbox0\hbox{$\textstyle\smile$}
\setbox1\hbox to \wd0{\hfil{\sl /\/}\hfil} \setbox2\hbox to
\wd0{\hfil\vrule height 10pt depth -2pt width
               \forklinewidth\hfil}
\newbox\doesforkbox
\setbox\doesforkbox\hbox{\lower 0pt\box1 \lower
2pt\box2\lower2pt\box0\relax}

\newcommand{\nf}{\unionstick}

\newcommand{\Sbs}{\gS^\text{bs}}
\newcommand{\s}{\mathfrak{s}}

\def\lta{<}
\def\lea{\le}

\def\gea{\ge}

\title[Symmetry in abstract elementary classes]{Symmetry in abstract elementary classes with amalgamation}
\date{\today\\
AMS 2010 Subject Classification: Primary 03C48. Secondary: 03C45, 03C52, 03C55, 03C75, 03E55.}
\keywords{Abstract elementary classes; Categoricity; Superstability; Tameness; Symmetry; Splitting; Good frame; Limit models; Saturated models}

\author{Monica M. VanDieren}

\address{Department of Mathematics\\
Robert Morris University\\
Moon Township Pennsylvania USA}
\email{vandieren@rmu.edu}
\urladdr{https://sites.google.com/a/rmu.edu/vandieren/home}

\author{Sebastien Vasey}

\address{Department of Mathematical Sciences, Carnegie Mellon University, Pittsburgh, Pennsylvania, USA}
\email{sebv@cmu.edu}
\urladdr{http://math.cmu.edu/\textasciitilde svasey/}
\thanks{This material is based upon work done while the second author was supported by the Swiss National Science Foundation under Grant No.\ 155136.}
\begin{document}

\begin{abstract}
  This paper is part of a program initiated by Saharon Shelah to extend the model theory of first order logic to the non-elementary setting of abstract elementary classes (AECs).  An abstract elementary class is a semantic generalization of the class of models of a complete first order theory with the elementary substructure relation.
 We examine the symmetry property of splitting (previously isolated by the first author) in AECs with amalgamation that satisfy a local definition of superstability.
 
 The key results are a downward transfer of symmetry and a deduction of symmetry from failure of the order property. 
  These results are then used to prove several structural properties in categorical AECs, improving classical results of Shelah who focused on the special case of categoricity in a successor cardinal.

  We also study the interaction of symmetry with tameness, a locality property for Galois (orbital) types. We show that superstability and tameness together imply symmetry. This sharpens previous work of Boney and the second author.
\end{abstract}

\maketitle

\tableofcontents

\section{Introduction}
The guiding conjecture for the classification of abstract elementary classes (AECs) is Shelah's categoricity conjecture.  For an introduction to AECs and Shelah's cateogicity conjecture, see \cite{baldwinbook09}.

Although most progress towards Shelah's categoricity conjecture has been made under the assumption that the categoricity cardinal is a successor, e.g.\ \cite{sh394, tamenessthree, tamelc-jsl}, recently, the second author has proved a categoricity transfer theorem without assuming that the categoricity cardinal is a successor, but assuming that the class is universal \cite{ap-universal-v10, categ-universal-2-v3-toappear} (other partial results not assuming categoricity in a successor cardinal are in \cite{indep-aec-apal} and \cite[Chapter IV]{shelahaecbook}). In this paper, we work in a more general framework than  \cite{ap-universal-v10, categ-universal-2-v3-toappear}.  We assume the amalgamation property and no maximal models and deduce new structural results without having to assume that the categoricity cardinal is a successor, or even has ``high-enough'' cofinality.

Beyond Shelah's categoricity conjecture, a major focus in developing a classification theory for AECs has been to find an appropriate generalization of first-order superstability. Approximations isolated in \cite{sh394} and \cite{shvi635} have provided a mechanism for proving categoricity transfer results (see also \cite{tamenessthree}, \cite{ap-universal-v10, categ-universal-2-v3-toappear}). In Chapter IV of \cite{shelahaecbook}, Shelah introduced \emph{solvability} and claims it should be the true definition of superstability in AECs (see Discussion 2.9 in the introduction to \cite{shelahaecbook}). It seems, however, that under the assumption that the class has amalgamation, a more natural definition is a version of ``$\kappa (T) = \aleph_0$'', first considered without the assumption of categoricity in \cite{gvv-mlq}. In \cite{gv-superstability-v5-toappear}, it is shown that this definition is equivalent to many others (including solvability and the existence of a \emph{good frame}, a local notion of independence), provided that the AEC satisfies a locality property for types called \emph{tameness} \cite{tamenessone}. 

Without tameness, progress has been made in the study of structural consequences of the Shelah-Villaveces definition of superstability such as the uniqueness of limit models (e.g. \cite{gvv-mlq}) or the property that the union of saturated models is saturated (\cite{bv-sat-v3,vandieren-chainsat-apal}). Recently in \cite{vandieren-symmetry-apal}, the first author isolated a symmetry property for splitting that turns out to be closely related to the uniqueness of limit models.

\subsection{Transferring symmetry}
In this paper we prove a downward transfer theorem for this symmetry property. This allows us to gain insight into all of the aspects of superstability mentioned above.

\begin{theorem}\label{transfer symmetry}
  Let $\K$ be an AEC. Suppose $\lambda$ and $\mu$ are cardinals so that $\lambda>\mu\geq\LS(\K)$ and $\K$ is superstable in every $\chi \in [\mu, \lambda]$. Then $\lambda$-symmetry implies $\mu$-symmetry.
\end{theorem}

Theorem \ref{transfer symmetry} (proven at the end of Section \ref{sym-transfer-sec}) improves Theorem 2 of \cite{vandieren-chainsat-apal} which transfers symmetry from $\mu^+$ to $\mu$. We also clarify the relationship between $\mu$-symmetry (as a property of $\mu$-splitting) and the symmetry property in good frames (see Section \ref{sym-props-sec}). The latter is older and has been studied in the literature: the work of Shelah in \cite{sh576} led to \cite[Theorem 3.7]{shelahaecbook}, which gives conditions under which a good frame (satisfying a version of symmetry) exists (but uses set-theoretic axioms beyond ZFC and categoricity in two successive cardinals). One should also mention \cite[Theorem IV.4.10]{shelahaecbook} which builds a good frame (in ZFC) from categoricity in a high-enough cardinal.  Note, however, the cardinal is very high and the underlying class of the frame is a smaller class of Ehrenfeucht-Mostowski models, although this can be fixed by taking an even larger cardinal.

It was observed in \cite[Theorem 5.14]{bgkv-apal} that Shelah's proof of symmetry of first-order forking generalizes naturally to give that the symmetry property of any reasonable global independence notion follows from the assumption of no order property. This is used in \cite{ss-tame-jsl} to build a good frame from tameness and categoricity (the results there are improved in \cite{indep-aec-apal, bv-sat-v3}). As for symmetry \emph{transfers}, Boney \cite{ext-frame-jml} has shown how to transfer symmetry of a good frame \emph{upward} using tameness for types of length two. This was later improved to tameness for types of length one with a more conceptual proof in \cite{tame-frames-revisited-v6-toappear}.

Theorem \ref{transfer symmetry} differs from these works in a few ways.  First, we do not assume tameness nor set-theoretic assumptions, and we do not work within the full strength of a frame or with categoricity (only with superstability). Also, we obtain a \emph{downward} and not an upward transfer. The arguments of this paper use \emph{towers} whereas the aforementioned result of Boney and the second author use independent sequences to transfer symmetry upward.

\subsection{Symmetry and superstability}
Another consequence of our work is a better understanding of the relationship between superstability and symmetry. It was claimed in an early version of \cite{gvv-mlq} that $\mu$-superstability directly implies the uniqueness of limit models of size $\mu$ but an error was later found in the proof. Here we show that this \emph{is} true provided we have enough instances of superstability:

\textbf{Theorem \ref{sym-from-superstab}.}
\textit{Let $\K$ be an AEC and let $\mu \ge \LS (\K)$. If $\K$ is superstable in all $\mu' \in [\mu, \beth_{\left(2^{\mu}\right)^+})$, then $\K$ has $\mu$-symmetry.}

The main idea is to imitate the proof of the aforementioned \cite[Theorem 5.14]{bgkv-apal} to get the order property from failure of symmetry. However we do not have as much global independence as there so the proof here is quite technical.

 \subsection{Implications in categorical AECs}
 
As a corollary of Theorem \ref{sym-from-superstab}, we obtain several applications to categorical AECs. A notable contribution of this paper is an improvement on a 1999 result of Shelah (see \cite[Theorem 6.5]{sh394}):

\begin{fact}
  Let $\K$ be an AEC with amalgamation and no maximal models. Let $\lambda$ and $\mu$ be cardinals such that $\cf(\lambda) > \mu > \LS (\K)$. If $\K$ is categorical in $\lambda$, then any limit model of size $\mu$ is saturated.
\end{fact}

Shelah claims in a remark immediately following his result that this can be generalized to show that for $M_0, M_1, M_2 \in \K_\mu$, if $M_1$ and $M_2$ are limit over $M_0$, then $M_1 \cong_{M_0} M_2$ (that is, the isomorphism also fixes $M_0$). This is however not what his proof gives (see the discussion after Theorem 10.17 in \cite{baldwinbook09}). Here we finally prove this stronger statement. Moreover, we can replace the hypothesis that $\cf (\lambda) > \mu$ by $\lambda \ge \beth_{\left(2^{\mu}\right)^+}$. That is, it is enough to ask for $\lambda$ to be high-enough (but of arbitrary cofinality):

\textbf{Corollary \ref{categoricity uniqueness corollary}.}
\textit{Let $\K$ be an AEC with amalgamation and no maximal models. Let $\lambda$ and $\mu$ are cardinals so that $\lambda>\mu\geq\LS(\K)$ and assume that $\K$ is categorical in $\lambda$. If either $\cf(\lambda) > \mu$ or $\lambda \ge \beth_{\left(2^{\mu}\right)^+}$, then whenever $M_0,M_1,M_2\in\K_\mu$ are such that both $M_1$ and $M_2$ are limit models over $M_0$, we have that $M_1\cong_{M_0}M_2$.}

  This gives a proof (assuming amalgamation, no maximal models, and a high-enough categoricity cardinal) of the (in)famous \cite[Theorem 3.3.7]{shvi635}, where a gap was identified in the first author's Ph.D.\ thesis. The gap was fixed assuming categoricity in $\mu^+$ in \cite{vandierennomax, nomaxerrata} (see also the exposition in \cite{gvv-mlq}). In \cite[Corollary 6.18]{bg-v11-toappear}, this was improved to categoricity in an arbitrary $\lambda > \mu$ provided that $\mu$ is big-enough and the class satisfies strong locality assumptions (full tameness and shortness and the extension property for coheir). In \cite[Theorem 7.11]{ss-tame-jsl}, only tameness was required but the categoricity had to be in a $\lambda$ with $\cf (\lambda) > \mu$. Still assuming tameness, this is shown for categoricity in any $\lambda \ge \beth_{(2^{\mu})^+}$ in \cite[Theorem 7.1]{bv-sat-v3}. Here assuming tameness we will improve this to categoricity in any $\lambda > \LS (\K)$ (see Corollary \ref{categ-tameness}).

In general, we obtain that an AEC with amalgamation categorical in a high-enough cardinal has several structural properties that were previously only known for AECs categorical in a cardinal of high-enough \emph{cofinality}, or even just in a successor. 

\begin{corollary}\label{categority-saturated-model-cor}
  Let $\K$ be an AEC with amalgamation. Let $\lambda > \mu \ge \LS (\K)$ and assume that $\K$ is categorical in $\lambda$.  Let $\mu \ge \LS (\K)$. If $\K$ is categorical in a $\lambda > \mu$, then:

  \begin{enumerate}
    \item\label{abstract-2} (see Corollary \ref{cor-mu-sat}) If $\lambda \ge \beth_{\left(2^{\mu}\right)^+}$ and $\mu > \LS (\K)$, then the model of size $\lambda$ is $\mu$-saturated.
    \item\label{abstract-3} (see Corollary \ref{good-frame-categ}) If $\mu \ge \beth_{(2^{\LS (\K)})^+}$ and the model of size $\lambda$ is $\mu^+$-saturated, then there exists a type-full good $\mu$-frame with underlying class the saturated models in $\K_\mu$.
  \end{enumerate}
\end{corollary}

  This improves several classical results from Shelah's milestone study of categorical AECs with amalgamation \cite{sh394}:

  \begin{itemize}
  \item Corollay \ref{categority-saturated-model-cor}.(\ref{abstract-2}) partially answers Baldwin \cite[Problem D.1.(2)]{baldwinbook09} which asked if in any AEC with amalgamation categorical in a high-enough cardinal, then the model in the categoricity cardinal is saturated.
  \item Corollay \ref{categority-saturated-model-cor}.(\ref{abstract-3}) partially answers the question in \cite[Remark 4.9.(1)]{sh394} of whether there is a parallel to forking in categorical AECs with amalgamation. It also improves on \cite[Theorem 7.4]{ss-tame-jsl}, which assumed categoricity in a successor (and a higher Hanf number bound).
  \item As part of the proof of Corollay \ref{categority-saturated-model-cor}.(\ref{abstract-3}), we derive \emph{weak tameness} (i.e.\ tameness over saturated models) from categoricity in a big-enough cardinal (this is Corollary \ref{weak-tameness-from-categ}). It was previously only known how to do so assuming that the categoricity cardinal has high-enough cofinality \cite[Main Claim II.2.3]{sh394}.
  \end{itemize}

We deduce a downward categoricity transfer in AECs with amalgamation (see also Corollary \ref{downward-categ}):

\textbf{Corollary \ref{fixed-point-downward}.}
\textit{Let $\K$ be an AEC with amalgamation. Let $\LS (\K) < \mu = \beth_\mu < \lambda$. If $\K$ is categorical in $\lambda$, then $\K$ is categorical in $\mu$.}

This improves on \cite[Theorem 10.16]{indep-aec-apal} where the result is stated with the additional assumption of $(<\mu)$-tameness.

\subsection{Implications in tame AECs}
This paper also combines our results with tameness: in Section \ref{sym-tame-sec}, we improve Hanf number bounds for several consequences of superstability. With Will Boney, the second author has shown \cite[Theorem 7.1]{bv-sat-v3} that $\mu$-superstability and $\mu$-tameness imply that for all high-enough $\lambda$, limit models of size $\lambda$ are unique (in the strong sense discussed above), unions of chains of $\lambda$-saturated models are saturated, and there exists a type-full good $\lambda$-frame. We transfer this behavior downward using our symmetry transfer theorem to get that the latter result is actually true starting from $\lambda = \mu^+$, and the former starting from $\lambda = \mu$:

\begin{corollary}\label{cor-tameness}
  Let $\mu \ge \LS (\K)$. If $\K$ is $\mu$-superstable and $\mu$-tame, then:

  \begin{enumerate}
    \item (see Corollary \ref{mu-symmetry-1}) If $M_0, M_1, M_2 \in \K_\mu$ are such that both $M_1$ and $M_2$ are limit models over $M_0$, then $M_1 \cong_{M_0} M_2$.
    \item (see Corollary \ref{chain-sat-1}) For any $\lambda > \mu$, the union of an increasing chain of $\lambda$-saturated models is $\lambda$-saturated.
    \item (see Corollary \ref{chain-sat-2}) There exists a type-full good $\mu^+$-frame with underlying class the saturated models in $\K_{\mu^+}$.
  \end{enumerate}
\end{corollary}

In fact, $\mu$-tameness along with $\mu$-superstability already implies $\mu$-symmetry. Many assumptions weaker than tameness (such as the existence of a good $\mu^+$-frame, see Theorem \ref{sym-good-frame}) suffice to obtain such a conclusion.

\subsection{Notes} A word on the background needed to read this paper: It is assumed that the reader has a solid knowledge of AECs (including the material in \cite{baldwinbook09}). Some familiarity with good frames, in particular the material of \cite{ss-tame-jsl} would be very helpful. In addition to classical results, e.g.\ in \cite{sh394}, the paper uses heavily the results of \cite{vandieren-symmetry-apal, vandieren-chainsat-apal} on limit models and the symmetry property of splitting. It also relies on the construction of a good frame in \cite{ss-tame-jsl}. At one point we also use the canonicity theorem for good frames \cite[Theorem 9.7]{indep-aec-apal}.

This paper was written while the second author was working on a Ph.D.\ thesis under the direction of Rami Grossberg at Carnegie Mellon University and he would like to thank Professor Grossberg for his guidance and assistance in his research in general and in this work specifically. We also thank the referees for reports which helped improve the presentation of this paper.



\section{Background}

All throughout this paper, we assume the amalgamation property:

\begin{hypothesis}\label{ap-hyp}
$\K$ is an AEC with amalgamation.
\end{hypothesis}

For convenience, we fix a big-enough monster model $\sea$ and work inside $\sea$.   This is possible since by Remark \ref{jep remark}, we will have the joint embedding property in addition to the amalgamation property for models of the relevant cardinalities. At some point, we will also use the following fact whose proof is folklore (see e.g.\ \cite[Proposition 10.13]{indep-aec-apal})

\begin{fact}\label{jep-decomp}
  Assume that $\K$ has joint embedding in some $\lambda \ge \LS (\K)$. Then there exists $\chi < \beth_{\left(2^{\LS (\K)}\right)^+}$ and an AEC $\K^\ast$ such that:

    \begin{enumerate}
    \item $\K^\ast \subseteq \K$ and $\K^\ast$ has the same strong substructure relation as $\K$.
    \item $\LS (\K^\ast) = \LS (\K)$.
    \item $\K^\ast$ has amalgamation, joint embedding, and no maximal models.
    \item $\K_{\ge \min (\lambda, \chi)} = (\K^\ast)_{\ge \min (\lambda, \chi)}$.
  \end{enumerate}
\end{fact}

Many of the pre-requisite definitions and notations used in this paper can be found in \cite{gvv-mlq}.  Here we recall the more specialized concepts that we will be using explicitly.  

We write $\gtp (\ba / M)$ for the Galois type of the sequence $\ba$ over $M$ (and we write $\gS (M)$ for the set of all Galois types over $M$). While the reader can think of $\gtp (\ba / M)$ as the orbit of $\ba$ under the action of $\Aut_{M} (\sea)$, $\gtp (\ba / M)$ is really defined as the equivalence class of the triple $(\ba, M, \sea)$ under a certain equivalence relation (see for example \cite[Definition 6.4]{grossberg2002}). This allows us to define the restriction of a Galois type to any strong substructure of its domain, as well as its image under any automorphism of $\sea$ (and by extension any $\K$-embedding whose domain contains the domain of the type).

With that remark in mind, we can state the definition of non-splitting, a notion of independence from \cite[Definition 3.2]{sh394}. 

\begin{definition}\label{def:splitting}
A type $p \in \gS (N)$ does not $\mu$-split over $M$ if and only if for any $N_1, N_2 \in \K_\mu$ such that $M \lea N_\ell \lea N$ for $\ell = 1,2$, and any $f: N_1 \cong_M N_2$, we have $f (p \rest N_1) = p \rest N_2$
\end{definition}

We will use the definition of universality from \cite[Definition I.2.1]{vandierennomax}:

\begin{definition}
  Let $M, N \in \K$ be such that $M \lea N$. We say that $N$ is \emph{$\mu$-universal over $M$} if for any $M' \gea M$ with $\|M'\| \le \mu$, there exists $f: M' \xrightarrow[M]{} N$. We say that $N$ is \emph{universal over $M$} if $N$ is $\|M\|$-universal over $M$.
\end{definition}

A fundamental concept in the study of superstable AECs is the notion of a \emph{limit model}, first introduced in \cite{sh394}. We only give the definition here and refer the reader to \cite{gvv-mlq} for more history and motivation.

\begin{definition}
  Let $\mu \ge \LS (\K)$ and let $\alpha < \mu^+$ be a limit ordinal. Let $M \in \K_{\mu}$. We say that \emph{$N$ is $(\mu, \alpha)$-limit over $M$} (or \emph{a $(\mu, \alpha)$-limit model over $M$}) if there exists a strictly increasing continuous chain $\seq{M_i : i \le \alpha}$ in $\K_{\mu}$ such that $M_0 = M$, $M_\alpha = N$, and $M_{i + 1}$ is universal over $M_i$ for all $i < \alpha$. We say that $N$ is \emph{limit over $M$} (or \emph{a limit model over $M$}) if it is $(\mu, \beta)$-limit over $M$ for some $\beta < \mu^+$. Finally, we say that $N$ is \emph{limit} if it is limit over $N_0$ for some $N_0 \in \K_{\|N\|}$.
\end{definition}

Towers were introduced in Shelah and Villaveces \cite{shvi635} as a tool to prove the uniqueness of limit models. A tower is an increasing sequence of length $\alpha$ of  limit models, denoted by $\bar M=\langle M_i\in\K_\mu\mid i<\alpha\rangle$, along with a sequence of designated elements $\bar a=\langle a_{i}\in M_{i+1}\backslash M_i\mid i+1<\alpha\rangle$ and a sequence of designated submodels $\bar N=\langle N_{i}\mid i+1<\alpha\rangle$ for which
 $N_i\lea M_{i}$, $\tp(a_i/M_i)$ does not $\mu$-split over $N_i$, and $M_i$ is universal over $N_i$  (see Definition I.5.1 of \cite{vandierennomax}). 

Now we recall a bit of terminology regarding towers.  The collection of all towers $(\bar M,\bar a,\bar N)$ made up of models of cardinality $\mu$ and sequences indexed by $\alpha$ is denoted by $\K^*_{\mu,\alpha}$.  For $(\bar M,\bar a,\bar N)\in\K^*_{\mu,\alpha}$, if $\beta<\alpha$ then we write $(\bar M,\bar a,\bar N)\restriction\beta$ for the tower made of the subseqences $\bar M\restriction\beta=\langle M_i\mid i<\beta\rangle$, $\bar a\restriction\beta=\langle a_i\mid i+1<\beta\rangle$, and $\bar N\restriction\beta=\langle N_i\mid i+1<\beta\rangle$.  We sometimes abbreviate the tower $(\bar M,\bar a,\bar N)$ by $\T$.

\begin{definition}\label{def:mu extension}

For towers $(\bar M,\bar a,\bar N)$ and $(\bar M',\bar a',\bar N')$ in $\K^*_{\mu,\alpha}$, we say $$(\bar M,\bar a,\bar N)\leq (\bar M',\bar a',\bar N')$$ if for all $i<\alpha$, $M_i\lea M'_i$, $\bar a=\bar a'$, $\bar N=\bar N'$ and whenever $M'_i$ is a proper extension of $M_i$, then $M'_i$ is universal over $M_i$.  If for each $i<\alpha$,  $M'_i $ is universal over $M_i$ we will write $(\bar M,\bar a,\bar N)< (\bar M',\bar a',\bar N')$.
\end{definition}

In order to transfer symmetry from $\lambda$ to $\mu$ we will need to consider a generalization of these towers where the models $M_i$ and $N_i$ may have different cardinalities.  Fix $\lambda\geq\mu\geq\LS(\K)$ and $\alpha$ a limit ordinal $<\mu^+$.  We will write $\K^*_{\lambda,\alpha,\mu}$ for the collection of towers of the form $(\bar M,\bar a,\bar N)$ where $\bar M=\langle M_i\mid i<\alpha\rangle$ is a sequence of models each of cardinality $\lambda$ and $\bar N=\langle N_i\mid i+1<\alpha\rangle$ is a sequence of models of cardinality $\mu$.  We require that for $i<\alpha$, $M_i$ is $\mu$-universal over $N_i$ and $\tp(a_i/M_i)$ does not $\mu$-split over $N_i$.  

In a natural way we order these towers by the following adaptation of Definition \ref{def:mu extension}.
\begin{definition}\label{tower-limit-order-def}
Let $\lambda\geq\chi\geq\mu\geq\LS(\K)$ be cardinals and fix $\alpha<\mu^+$ an ordinal.
For towers $(\bar M,\bar a,\bar N)\in\K^*_{\lambda,\alpha,\mu}$ and $(\bar M',\bar a',\bar N')\in\K^*_{\chi,\alpha,\mu}$, we say $$(\bar M,\bar a,\bar N)<_\mu (\bar M',\bar a',\bar N')$$ if for all $i<\alpha$, $M_i\lea M'_i$, $\bar a=\bar a'$, $\bar N=\bar N'$, and there is $\theta<\lambda^+$ so that $M'_i$ is  a $(\lambda,\theta)$-limit model witnessed by a sequence $\langle M^\gamma_i\mid\gamma < \theta\rangle$ with $M_i \lta M^\gamma_0$.
\end{definition}

Note that Definition \ref{def:mu extension} is defined only on towers in $\K^*_{\mu,\alpha}$ and is slightly weaker from the ordering $<_\mu$ when restricted to $\K^*_{\mu,\alpha}$.  In particular, the models $M'_i$ in the tower $(\bar M',\bar a,\bar N)$ $<$-extending $(\bar M,\bar a,\bar N)$ are only required to be universal over $M_i$ and limit.  It is not necessary that $M'_i$ is limit over $M_i$ as we require if $(\bar M',\bar a,\bar N)<_\mu (\bar M,\bar a,\bar N)$.

Towers are particularly suited for superstable abstract elementary classes, in which they are known to exist and in which the union of an increasing chain of towers will be a tower. The definition below is already implicit in \cite{shvi635} and has since then been studied in many papers, e.g.\ \cite{vandierennomax, gvv-mlq, indep-aec-apal, bv-sat-v3, gv-superstability-v5-toappear}. We will use the definition from \cite[Definition 10.1]{indep-aec-apal}:

\begin{definition}\label{ss assm}
$\K$ is \emph{$\mu$-superstable} (or \emph{superstable in $\mu$})
if:

  \begin{enumerate}
    \item $\mu \ge \LS (\K)$.
    \item $\K_\mu$ is nonempty, has joint embedding, and no maximal models.
    \item $\K$ is stable in $\mu$.  That is, $|\gS (M)| \le \mu$ for all $M \in \K_\mu$. Some authors call this ``Galois-stable," and:
    \item$\mu$-splitting in $\K$ satisfies the ``no long splitting chains'' property:
  
      For any limit ordinal $\alpha < \mu^+$, for every sequence $\langle M_i\mid i<\alpha\rangle$ of
      models of cardinality $\mu$ with $M_{i+1}$ universal over $M_i$ and for every $p\in\gaS(\bigcup_{i < \alpha} M_i)$, there exists $i<\alpha$ such that $p$ does not $\mu$-split over $M_i$.
\end{enumerate}
\end{definition}
\begin{remark}\label{jep remark}
  By our global hypothesis of amalgamation (Hypothesis \ref{ap-hyp}), if $\K$ is $\mu$-superstable, then $\K_{\ge \mu}$ has joint embedding.
\end{remark}
\begin{remark}
  By the weak transitivity property of $\mu$-splitting \cite[Proposition 3.7]{ss-tame-jsl}, $\mu$-superstability implies the following continuity property (which is sometimes also stated as part of the definition): For any limit ordinal $\alpha < \mu^+$, for every sequence $\seq{M_i \mid i < \alpha}$ of models of cardinality $\mu$ with $M_{i + 1}$ universal over $M_i$ and for every $p \in \gS (\bigcup_{i < \alpha} M_i)$, if $p \rest M_i$ does not $\mu$-split over $M_0$ for all $i < \alpha$, then $p$ does not $\mu$-split over $M_0$. We will use this freely.
\end{remark}

The main results of this paper involve the concept of symmetry over limit models and its equivalent formulation involving towers which was identified in \cite{vandieren-symmetry-apal}:
\begin{definition}\label{sym defn}
We say that an abstract elementary class  exhibits \emph{symmetry for non-$\mu$-splitting} if  whenever models $M,M_0,N\in\K_\mu$ and elements $a$ and $b$  satisfy the conditions \ref{limit sym cond}-\ref{last} below, then there exists  $M^b$  a limit model over $M_0$, containing $b$, so that $\tp(a/M^b)$ does not $\mu$-split over $N$.  See Figure \ref{fig:sym}.
\begin{enumerate} 
\item\label{limit sym cond} $M$ is universal over $M_0$ and $M_0$ is a limit model over $N$.
\item\label{a cond}  $a\in M\backslash M_0$.
\item\label{a non-split} $\tp(a/M_0)$ is non-algebraic and does not $\mu$-split over $N$.
\item\label{last} $\tp(b/M)$ is non-algebraic and does not $\mu$-split over $M_0$. 
   
\end{enumerate}
\end{definition}

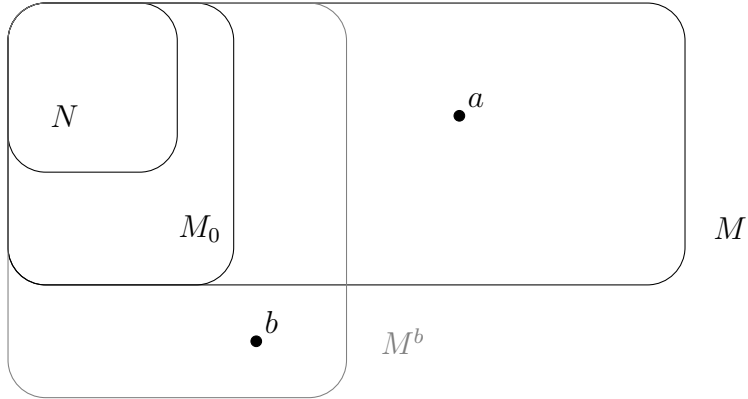
\begin{figure}[h]
\begin{tikzpicture}[rounded corners=5mm, scale=3,inner sep=.5mm]
\draw (0,1.25) rectangle (.75,.5);
\draw (.25,.75) node {$N$};
\draw (0,0) rectangle (3,1.25);
\draw (0,1.25) rectangle (1,0);
\draw (.85,.25) node {$M_0$};
\draw (3.2, .25) node {$M$};
\draw[color=gray] (0,1.25) rectangle (1.5, -.5);
\node at (1.1,-.25)[circle, fill, draw, label=45:$b$] {};
\node at (2,.75)[circle, fill, draw, label=45:$a$] {};
\draw[color=gray] (1.75,-.25) node {$M^{b}$};
\end{tikzpicture}
\caption{A diagram of the models and elements in the definition of symmetry. We assume the type $\tp(b/M)$ does not $\mu$-split over $M_0$ and $\tp(a/M_0)$ does not $\mu$-split over $N$.  Symmetry implies the existence of $M^b$ a limit model over $M_0$ containing $b$, so that $\tp(a/M^b)$  does not $\mu$-split over $N$.} \label{fig:sym}
\end{figure}

We end by recalling a few results of the first author showing the importance of the symmetry property:

\begin{fact}[Theorem 2 in \cite{vandieren-symmetry-apal}]\label{chainsat-sym}
  If $\K$ is $\mu$-superstable and the union of any chain (of length less than $\mu^{++}$) of saturated models of size $\mu^+$ is saturated, then $\K$ has $\mu$-symmetry.
\end{fact}

Many of the results on symmetry rely on the equivalent formulation of $\mu$-symmetry in terms of reduced towers.
 
\begin{definition}
A  tower $(\bar M,\bar a,\bar N)\in\K^*_{\mu,\alpha}$ is \emph{reduced} if it satisfies the condition that for every $<$-extension $(\bar M',\bar a,\bar N)\in\K^*_{\mu,\alpha}$ of $(\bar M,\bar a,\bar N)$ and for every $i<\alpha$, $M'_i\bigcap(\Union_{j<\alpha}M_j) = M_i$.
\end{definition}
\begin{definition}
  A tower $(\bar M, \bar a, \bar N) \in \K^*_{\mu, \alpha}$ is \emph{continuous} if for any limit $i < \alpha$, $M_i = \bigcup_{j < i} M_j$.
\end{definition}

\begin{fact}[Theorem 3 in \cite{vandieren-symmetry-apal}]\label{sym-reduced-tower}
  Assume $\K$ is $\mu$-superstable. The following are equivalent:

  \begin{enumerate}
    \item\label{symmetry item} $\K$ has $\mu$-symmetry.
    \item\label{reduced are continuous} Any reduced tower in $\K_{\mu, \alpha}^\ast$ is continuous.        \end{enumerate}
\end{fact}

In \cite{shvi635}, Shelah and Villaveces attempt to show that that the continuity of reduced towers gives the uniqueness of limit models using full towers.  Later Grossberg, VanDieren, and Villaveces verified  this statement through the introduction of relatively full towers \cite{gvv-mlq}:

\begin{fact}\label{uq-limit}
  Assume $\K$ is $\mu$-superstable. If any reduced tower in $\K_{\mu, \alpha}^\ast$ is continuous (or equivalently by Fact \ref{sym-reduced-tower} if $\K$ has $\mu$-symmetry), then for any $M_0, M_1, M_2 \in \K_\mu$, if $M_1$ and $M_2$ are limit over $M_0$, then $M_1 \cong_{M_0} M_2$.
\end{fact}

Symmetry also has implications to chains of saturated models. For $\lambda > \LS (\K)$, write $\Ksatp{\lambda}$ for the class of $\lambda$-saturated models in $\K_{\ge \lambda}$. We also define $\Ksatp{0} := \K$. Using this notation, we have:

\begin{fact}[Theorem 22 in \cite{vandieren-chainsat-apal}]\label{union-sat-monica}
  Assume $\K$ is $\mu$-superstable, $\mu^+$-superstable, and every limit model in $\K_{\mu^+}$ is saturated. Then $\Ksatp{\mu^+}$ is an AEC with $\LS (\Ksatp{\mu^+}) = \mu^+$.
\end{fact}
\begin{remark}\label{sym-chainsat-rmk}
  By Fact \ref{uq-limit}, the hypotheses of Fact \ref{union-sat-monica} hold if $\K$ is $\mu$-superstable, $\mu^+$-superstable, and has $\mu^+$-symmetry.
\end{remark}

We will also use the following easy lemma:

\begin{lemma}\label{chainsat-lim}
  Let $\lambda$ be a limit cardinal and let $\lambda_0 < \lambda$. Assume that for all $\mu \in [\lambda_0, \lambda)$, $\Ksatp{\mu}$ is an AEC with $\LS (\Ksatp{\mu}) = \mu$. Then $\Ksatp{\lambda}$ is an AEC with $\LS (\Ksatp{\lambda}) = \lambda$.
\end{lemma}
\begin{proof}
  That $\Ksatp{\lambda}$ is closed under chains is easy to check. To see $\LS (\Ksatp{\lambda}) = \lambda$, let $M \in \Ksatp{\lambda}$ and let $A \subseteq |M|$. Without loss of generality, $\chi := |A| \ge \lambda$. Let $\delta := \cf(\lambda)$ and let $\seq{\lambda_i : i < \delta}$ be an increasing sequence of cardinals with limit $\lambda$. Build $\seq{M_i : i \le \delta}$ increasing continuous in $\K_{\chi}$ such that for all $i < \delta$, $M_{i + 1}$ is $\lambda_i^+$-saturated and $A \subseteq |M_0|$. This is possible by assumption. Then $M_\delta$ is $\lambda_i^+$-saturated for all $i < \delta$, hence is $\lambda$-saturated. Thus it is as needed.
\end{proof}

\section{Transferring symmetry}\label{sym-transfer-sec}
In this section we prove Theorem \ref{transfer symmetry} which is key to the results in the following sections.  We start with a few observations which will allow us to extend the tower machinery from \cite{gvv-mlq} and \cite{vandieren-chainsat-apal} to include towers composed of models of different cardinalities. In particular, we derive an extension property for towers of different cardinalities, Lemma \ref{extension lemma}.  This will allow us to adapt the arguments from \cite{vandieren-chainsat-apal} to prove Theorem \ref{transfer symmetry}.

We start with a study of chains where each model indexed by a successor is universal over its predecessor:

\begin{proposition}\label{limit model lemma}
Suppose that $\lambda \ge \LS (\K)$ is a cardinal. Assume that $\K$ is stable in $\lambda$ with no maximal models of cardinality $\lambda$. Let $\theta$ be a limit ordinal. Assume $\langle M_i \in\K_{\ge \LS (\K)}\mid i<\theta\rangle$ is a strictly increasing and continuous sequence of models so that for all $i < \theta$, $M_{i+1}$ is universal over $M_i$. If $M := \Union_{i<\theta}M_i$ has size $\lambda$, then $M$ is a $(\lambda,\theta)$-limit model over some model containing $M_0$.
\end{proposition}
\begin{proof}
  By cardinality considerations, $\theta < \lambda^+$. Replacing $\theta$ by $\cf(\theta)$ if necessary, we can assume without loss of generality that $\theta$ is regular. By $\lambda$-stability and the assumption that $\K$ has no maximal models of cardinality $\lambda$, we can fix a $(\lambda, \theta)$-limit model $M^*$ witnessed by $\langle M^*_i\mid i<\theta\rangle$ with $M_0\lea M^*_0$. If there exists $i < \theta$ such that $M_i \in \K_\lambda$, then the sequence $\langle M_j \mid j \in [i, \theta) \rangle$ witnesses that $M$ is $(\lambda, \theta)$-limit and $M_0 \lea M_i$; so assume that $\lambda > \LS (\K)$ and $M_i \in K_{<\lambda}$ for all $i < \theta$. Then we must have that $\theta = \cf (\lambda)$. If $\lambda$ is a successor, we must have that $\theta = \lambda$ and we obtain the result from \cite[Proposition 14]{vandieren-chainsat-apal}; so assume $\lambda$ is limit. For $i < \theta$, let $\lambda_i := \|M_i\|$.

 Fix $\langle a_\alpha\mid \alpha<\lambda\rangle$ an enumeration of $M^*$.  Using the facts that $M_{i+1}$ is universal over $M_i$ and that $M^*_{i+1}$ is universal over $M^*_i$,  we can build an isomorphism $f:M\cong M^*$ inductively by defining an increasing and continuous sequence of $\K$-embeddings $f_i$ so that
 $f_i:M_i\rightarrow M^*_i$, $f_0=\id_{M_0}$,
and $\{a_\alpha\mid \alpha<\lambda_i\}\subseteq \rg(f_{i+1})$.
\end{proof}

We will use the following generalization of the weak transitivity property of $\mu$-splitting proven in \cite[Proposition 3.7]{ss-tame-jsl}. The difference here is that the models are allowed to be of size bigger than $\mu$.

\begin{proposition}\label{weak-trans}
  Let $\mu \ge \LS (\K)$ be such that $\K$ is stable in $\mu$. Let $M_0 \lea M_1 \lta M_1' \lea M_2$ all be in $\K_{\ge \mu}$. Assume that $M_1'$ is universal over $M_1$.
  Let $p \in \gS (M_2)$. If $p \rest M_1'$ does not $\mu$-split over $M_0$ and $p$ does not $\mu$-split over some $N \in \K_\mu$ with $N \lea M_1$, then $p$ does not $\mu$-split over $M_0$.
\end{proposition}
\begin{proof}
  Note that by definition of $\mu$-splitting, $M_0 \in \K_\mu$. Thus by making $N$ larger if necessary we can assume that $M_0 \lea N$. By basic properties of universality we have that $M_1'$ is universal over $N$, hence without loss of generality $M_1 = N$. In particular, $M_1 \in \K_\mu$. By stability, build $M_1'' \in \K_\mu$ universal over $M_1$ such that $M_1 \lta M_1'' \lea M_1'$. By monotonicity, $p \rest M_1''$ does not $\mu$-split over $M_0$. Thus without loss of generality also $M_1' \in \K_\mu$. By definition of $\mu$-splitting, it is enough to check that $p \rest M_2'$ does not $\mu$-split over $M_0$ for all $M_2' \in \K_\mu$ with $M_2' \lea M_2$. Thus without loss of generality again $M_2 \in \K_\mu$. Now use the weak transitivity property of $\mu$-splitting \cite[Proposition 3.7]{ss-tame-jsl}.
\end{proof}

We use the previous proposition to extend the continuity property of $\mu$-splitting to models of size bigger than $\mu$. This is very similar to the argument in \cite[Claim II.2.11]{shelahaecbook}.

\begin{proposition}\label{limit splitting proposition}
  Let $\mu \ge \LS (\K)$ and assume that $\K$ is $\mu$-superstable. 

  Suppose $\langle M_i\in\K_{\ge \mu}\mid i< \delta\rangle$ is an increasing sequence of models so that, for all $i < \delta$, $M_{i+1}$ is universal over $M_i$. Let $p\in\gaS(\Union_{i<\delta}M_i)$.
  If $p\restriction M_i$ does not $\mu$-split over $M_0$ for each $i<\delta$, then $p$ does not $\mu$-split over $M_0$.
\end{proposition}
\begin{proof}
  Without loss of generality, $\delta = \cf (\delta)$. Let $M_\delta := \bigcup_{i < \delta} M_i$. There are two cases to check.  If $\delta>\mu$, then by \cite[Claim 3.3]{sh394}, there exists $N \in \K_\mu$ with $N \lea M_\delta$ such that $p$ does not $\mu$-split over $N$. Pick $i < \delta$ such that $N \lea M_i$. Then $p$ does not $\mu$-split over $M_i$. By Proposition \ref{weak-trans} (where $(M_0, M_1, M_1', M_2, N)$ there stand for $(M_0, M_i, M_{i + 1}, M_\delta, N)$ here), $p$ does not $\mu$-split over $M_0$.

  Suppose then that $\delta\leq\mu$ and for sake of contradiction that $M^*$ of cardinality $\mu$ witnesses the splitting of $p$ over $M_0$, i.e.\ $p\restriction M^*$ $\mu$-splits over $M_0$.  We can find $\langle M^*_i\in\K_\mu\mid i <\delta\rangle$ an increasing resolution of $M^*$ so that 
$M^*_i \lea M_i$ for all $i < \delta$.  By monotonicity of splitting, stability in $\mu$, and the fact that each $M_{i+1}$ is universal over $M_i$, we can increase $M^*$, if necessary, to 
arrange that $M^*_{i+1}$ is universal over $M^*_i$.  Since $p\restriction M_i$ does not $\mu$-split over $M_0$, monotonicity of non-splitting implies that $p\restriction M^*_i$ does not $\mu$-split over $M_0$.  Then, by $\mu$-superstability $p\restriction M^*$-does not $\mu$-split over $M_0$.  This contradicts our choice of $M^*$.

\end{proof}

We adapt the proof of  the extension property for non-$\mu$-splitting (\cite[Theorem I.4.10]{vandierennomax}) to handle models of different sizes under the additional assumption of superstability in the size of the bigger model. The conclusion can also be achieved using the assumption of tameness instead of superstability (since $\mu$-splitting and $\lambda$-splitting coincide if $\K$ is $\mu$-tame and $\mu \le \lambda$, see \cite[Proposition 3.12]{bgkv-apal}).

\begin{proposition}\label{extension for non-splitting}
Fix cardinals $\lambda>\mu\geq\LS(\K)$. Suppose that $\K$ is $\mu$-stable and $\lambda$-superstable.

  Let $M \in \K_\mu$ and $M^\lambda, M' \in\K_\lambda$ be such that $M \lea M^\lambda \lea M'$ and $M^\lambda$ is limit over some model containing $M$. Let $p \in \gS (M^\lambda)$ be such that $p$ does not $\mu$-split over $M$. Then there exists $q \in \gS (M')$ extending $p$ so that $q$ does not $\mu$-split over $M$. Moreover $q$ is algebraic if and only if $p$ is.
\end{proposition}
\begin{proof}
  Let $\theta < \lambda^+$ and $\langle M^\lambda_i\mid i<\theta\rangle$ witness that $M^\lambda$ is $(\lambda, \theta)$-limit with $M \lea M_0^\lambda$. Write $p:=\gtp(a/M^\lambda)$.
By $\lambda$-superstability there exists $i<\theta$ so that $p$ does not $\lambda$-split over $M^\lambda_i$.  Since $M^\lambda_{i+2}$ is universal over $M^\lambda_{i+1}$ there exists $f:M'\underset{M^\lambda_{i+1}} \rightarrow M^\lambda_{i+2}$. Extend $f$ to $g \in \Aut_{M_{i + 1}^\lambda} (\sea)$. Let $q := g^{-1} (p) \rest M' = \gtp (g^{-1} (a) / M')$. Note that $q$ is nonalgebraic if $p$ is nonalgebraic (the converse will follow once we have shown that $q$ extends $p$). By monotonicity, invariance, and our assumption that $p$ does not $\mu$-split over $M$, we can conclude that $q$ does not $\mu$-split over $M$. By similar reasoning also $q$ does not $\lambda$-split over $M^\lambda_i$. 
In particular $\gtp(g^{-1}(a)/M^\lambda) = q \rest M^\lambda$ does not $\lambda$-split over $M^\lambda_i$.  Since $g$ fixes $M^\lambda_{i+1}$, we know that $g^{-1}(a)$ realizes $p\restriction M^\lambda_{i+1}$. 
Therefore, we get by the uniqueness of non-$\lambda$-splitting extensions that
$q \rest M^\lambda = \gtp(f^{-1}(a)/M^\lambda)=\gtp(a/M^\lambda) = p$. This shows that $q$ extends $p$, as desired.
\end{proof}

We can now prove an extension property for towers in $\K^*_{\lambda,\alpha,\mu}$.

\begin{lemma}\label{extension lemma 0}
  Let $\lambda$ and $\mu$ be cardinals satisfying $\lambda \ge \mu \ge \LS (\K)$. Assume that $\K$ is superstable in $\mu$ and in $\lambda$. For any $(\bar M, \bar a, \bar N) \in \K_{\lambda, \alpha, \mu}^\ast$, there exists $(\bar M', \bar a, \bar N) \in \K_{\lambda, \alpha, \mu}^\ast$ so that:

  $$
  (\bar M,\bar a,\bar N)<_{\mu}(\bar M',\bar a,\bar N)
  $$
\end{lemma}
\begin{proof}
    If $\lambda = \mu$, the result follows from infinitely many (for example $\cf (\lambda)$ many) applications of \cite[Lemma 5.3]{gvv-mlq} which is the extension property for towers. If $\lambda > \mu$, the result follows similarly from the proof of the extension property for towers using Proposition \ref{extension for non-splitting}. 
\end{proof}

We also have a continuity property:

\begin{lemma}\label{continuity lemma}
  Let $\mu \ge \LS (\K)$ be such that $\K$ is $\mu$-superstable. Let $\seq{\lambda_i : i < \delta}$ be an increasing sequence of cardinals with $\lambda_0 \ge \mu$. Let $\seq{(\bar{M}^i, \bar{a}, \bar{N}) \in \K_{\lambda_i, \alpha, \mu}^* \mid i < \delta}$ be a sequence of towers such that $(\bar{M}^i, \bar{a}, \bar{N}) \lta_{\mu} (\bar{M}^{i + 1}, \bar{a}, \bar{N})$ for all $i < \delta$.

Let $\bar{M}^\delta$ be the sequence composed of models of the form $M_\beta^{\delta} := \bigcup_{i < \delta} M_\beta^{i}$ for $\beta < \alpha$. Let $\lambda := \sum_{i < \delta}{\lambda_i}$. 

Then $(\bar{M}^\delta, \bar{a}, \bar N) \in \K_{\lambda, \alpha, \mu}*$ and $(\bar{M}^i, \bar{a}, \bar{N}) \lta_{\mu}(\bar{M}^\delta, \bar{a}, \bar N)$ for all $i < \delta$.
\end{lemma}
\begin{proof}
  Working by induction on $\delta$, we can assume without loss of generality that the sequence of tower is continuous. That is, for each $\beta < \alpha$ and limit $i < \delta$, $M_\beta^i = \bigcup_{j < i} M_\beta^j$. Of course, it is enough to show that $(\bar{M}^0, \bar{a}, \bar{N}) \lta_{\mu} (\bar{M}^\delta, \bar{a}, \bar N)$. Let $\beta < \alpha$. There are two things to check: $M^\lambda_\beta$ is a limit model over a model that contains $M_\beta^0$, and $\gtp(a_\beta/M^\lambda_\beta)$ does not $\mu$-split over $N_\beta$. Proposition \ref{limit model lemma} confirms that $M^\lambda_\beta$ is a $(\lambda,\delta)$-limit model over some model containing $M_\beta^0$.  Because each $(\bar M^i,\bar a,\bar N)$ is a tower, we know that $\gtp(a_\beta/M^i_\beta)$ does not $\mu$-split over $N_\beta$.  This allows us to apply Proposition \ref{limit splitting proposition} to conclude that $\gtp(a_\beta/M^\lambda_\beta)$ does not $\mu$-split over $N_\beta$.
\end{proof}

We conclude an extension property for towers of different sizes

\begin{lemma}\label{extension lemma}
Let $\kappa$, $\lambda$ and $\mu$ be cardinals satisfying $\lambda\geq\kappa\geq\mu\geq\LS(\K)$. Assume that $\K$ is superstable
in $\mu$ and in every $\chi \in [\kappa, \max(\kappa^+, \lambda))$.

  Let $(\bar M^\kappa,\bar a,\bar N)\in\K^*_{\kappa,\alpha,\mu}$.
\begin{enumerate}
\item There exists $(\bar M,\bar a,\bar N)\in\K^*_{\lambda,\alpha,\mu}$ so that $$(\bar M^\kappa,\bar a,\bar N)<_{\mu}(\bar M,\bar a,\bar N).$$
\item\label{furthermore lemma} 
If in addition $\K$ is $\lambda$-superstable, then there exists a sequence $\langle N^\lambda_\beta\mid \beta<\alpha\rangle$ so that $N_\beta\leq N^\lambda_\beta$ for all $\beta < \alpha$ and $(\bar M,\bar a,\bar N^\lambda)\in\K^*_{\lambda,\alpha}$ (so $\|N_\beta^\lambda\| = \lambda$ for all $\beta < \alpha$).

\end{enumerate}
\end{lemma}

\begin{proof}
We prove the first statement in the lemma by induction on $\lambda$. If $\lambda = \kappa$, this is given by Lemma \ref{extension lemma 0}. Now assume that $\lambda > \kappa$. Fix an increasing continuous sequence $\seq{\lambda_i \mid i < \cf (\lambda)}$ which is cofinal in $\lambda$ and so that $\lambda_0 = \kappa$ (if $\lambda = \chi^+$ is a successor we can take $\lambda_i = \chi$ for all $i < \lambda$). We build a sequence $\seq{(\bar{M}^{\lambda_i}, \bar{a}, \bar{N}) \in \K_{\lambda_i, \alpha, \mu}^* \mid i < \cf (\lambda)}$ which is increasing (that is, $(\bar{M}^{\lambda_i}, \bar{a}, \bar{N}) \lta_{\mu} (\bar{M}^{\lambda_{i + 1}}, \bar{a}, \bar{N})$ for all $i < \cf (\lambda)$) and continuous (in the obvious sense, see Lemma \ref{continuity lemma}). This is possible by the induction hypothesis. Now by Lemma \ref{continuity lemma}, the union of the chain of towers (defined there) is as desired.

For part (\ref{furthermore lemma}), recall from Definition \ref{tower-limit-order-def} that for each $\beta < \alpha$, $M_\beta$ is a limit model over some model containing $M_\beta^\kappa$.  Let $\langle M^*_{\beta,i}\in\K_\lambda\mid i< \theta_{\beta}\rangle$ witness this. By $\lambda$-superstability, for each $\beta<\alpha$, there exists $i_\beta<\theta_{\beta}$ so that $\gtp(a_\beta/M^\lambda_\beta)$ does not $\lambda$-split over $M^*_{\beta,i_\beta}$.  By our choice of $M^*_{\beta,0}$ containing $M_\beta^\kappa$, and consequently $N_\beta$, we can take $N^\lambda_\beta:=M^*_{\beta,i_\beta}$.  
\end{proof}

We now begin the proof of Theorem \ref{transfer symmetry}.  The structure of the proof is similar to the proof of Theorem 2 of \cite{vandieren-chainsat-apal}; only here we work with towers in $\K^*_{\lambda,\alpha,\mu}$ as opposed to only towers in $\K^*_{\mu,\alpha}$.
\begin{proof}[Proof of Theorem \ref{transfer symmetry}]
Suppose for the sake of contradiction that $\K$ does not have symmetry for $\mu$-non-splitting.  By Fact \ref{sym-reduced-tower} and our $\mu$-superstability assumption, $\K$ has a reduced discontinuous tower in $\K^*_{\mu,\alpha}$ for some $\alpha<\mu^+$.  Let $\alpha$ be the minimal ordinal  for which there is a reduced, discontinuous tower  in $\K^*_{\mu,\alpha}$.  By Lemma 5.7 of \cite{gvv-mlq}, we may assume that $\alpha=\delta+1$ for some limit ordinal $\delta$.  Fix $\T=(\bar M,\bar a,\bar N)\in\K^*_{\mu,\alpha}$ a reduced discontinuous tower with $b\in M_\delta\backslash \Union_{\beta<\alpha}M_\beta$.    

Let $I := \cf(\lambda)$. By Lemma \ref{extension lemma}, we can build an increasing and continuous chain of  towers $\langle \T^i\mid i\in I\rangle$ extending $\T\restriction\delta$.  If $\lambda=\kappa^+$ for some $\kappa$, then select each $\T^i\in\K^*_{\kappa,\delta,\mu}$.  If $\lambda$ is a limit cardinal, fix $\langle \lambda_i\mid i<\cf(\lambda)\rangle$ to be an increasing and continuous sequence of cardinals cofinal in $\lambda$, with $\lambda_0>\mu$ and choose $\T^i\in\K^*_{\lambda_i,\delta,\mu}$.  Let $\T^\lambda:=\Union_{i\in I}\T^i$. 

Notice that by Lemma \ref{continuity lemma}, and our assumptions on the towers $\T^i$, we can conclude that $\T^\lambda\in\K^*_{\lambda,\delta,\mu}$ and $\T^\lambda$ extends $\T\restriction\delta$.  In particular, for each $\beta<\alpha$,
\begin{equation}\label{mu split equation}
\tp(a_\beta/M^\lambda_\beta)\text{ does not }\mu\text{-split over }N_\beta.
\end{equation}

Furthermore by the second part of Lemma \ref{extension lemma} we can find $N^\lambda_\beta$ so that the tower  defined by $(\bar M^\lambda,\bar a,\bar N^\lambda)$ is in $\K^*_{\lambda,\delta}$ and each $M^\lambda_\beta$ is a limit over $N^\lambda_\beta$. We can extend this to a tower of length $\delta+1$ by appending to $(\bar M^\lambda,\bar a,\bar N^\lambda)$ a model $M^\lambda_\delta$ of cardinality $\lambda$ containing $\Union_{\beta<\delta}M^\lambda_\beta$ and $M^\delta$.  Call this tower $\T^b$, since it contains $b$.

By $\lambda$-symmetry and Fact \ref{sym-reduced-tower}, we know that all reduced towers in $\K^*_{\lambda,\alpha}$ are continuous.  Therefore $\T^b$ is not reduced.  However, by the density of reduced towers \cite[Theorem 5.6]{gvv-mlq}, we can find a reduced, continuous extension of $\T^b$ in $\K^*_{\lambda,\delta+1}$.  By $\lambda$-many applications of this theorem, we may assume that for each $\beta<\alpha$, the model indexed by $\beta$ in this reduced tower is a $(\lambda,\cf(\lambda))$-limit over $M^\lambda_\beta$.
Refer to this tower as $\T^*$.    
See Fig. \ref{fig:tower}.
\begin{figure}[h]
\begin{tikzpicture}[rounded corners=5mm,scale =2.7,inner sep=.35mm]
\draw (0,1.5) rectangle (.75,.5);
\draw (0,1.5) rectangle (1.75,1);
\draw (.25,.65) node {$N_0$};
\draw (1.25,1.1) node {$N_\beta$};
\draw[rounded corners=5mm]  (0, 1.5) --(0,.75)-- (1,-1) -- (1.5,-1)  -- (1.75,1)--(1.75,1.5)--  cycle;
\draw (1.85,.7) node {$N^{\lambda}_\beta$};
\draw (0,0) rectangle (4,1.5);
\draw (.85,.25) node {$M_0$};
\draw(1.4,.25) node {$M_1$};
\draw (1.8,.25) node {$\dots M_\beta$};
\draw (2.35,.25) node {$M_{\beta+1}$};
\draw (3.15,.2) node {$\dots\displaystyle{\Union_{k<\delta}M_k}$};
\draw (3.85, .25) node {$M_\delta$};
\draw (-.4,.25) node {$\T\in\K^*_{\mu,\alpha}$};
\draw (0,1.5) rectangle (3.5, -.4);
\draw[rounded corners=5mm]  (0, 1.5) -- (0,-2) -- (3.6,-2)  -- (4,0)--(4,1.5) --  cycle;
\draw[rounded corners=5mm]  (0, 1.5) -- (0,-1.35) -- (3.5,-1.35)  -- (4,0)--(4,1.5) --  cycle;
\draw (.85,-.15) node {$M^{i}_0$};
\draw (1.8,-.15) node {$\dots M^{i}_\beta$};
\draw (2.35,-.15) node {$M^{i}_{\beta+1}$};
\draw(1.4,-.15) node {$M^{i}_1$};
\draw (3.15,-.2) node {$\dots\displaystyle{\Union_{l<\delta}M^{i}_l}$};
\draw (-.4,-.15) node {$\T^i\in\K^*_{\lambda_i,\delta,\mu}$};
\draw (.85,-.6) node {$\vdots$};
\draw (1.75,-.6) node {$\vdots$};
\draw (2.35,-.6) node {$\vdots$};
\draw (3.2,-.6) node {$\vdots$};
\draw (1.35,-.6) node {$\vdots$};
\draw (0,1.5) rectangle (3.5, -1.35);
\draw (0,1.5) rectangle (1,-2);
\draw(0,1.5) rectangle (1.5, -2);
\draw (0,1.5) rectangle (2.5, -2);
\draw (0,1.5) rectangle (2,-2);
\draw (.8,-1.15) node {$M^{\lambda}_0$};
\draw (1.8,-1.15) node {$ M^{\lambda}_\beta$};
\draw (2.3,-1.15) node {$M^{\lambda}_{\beta+1}$};
\draw(1.35,-1.15) node {$M^{\lambda}_1$};
\draw (3.1,-1.2) node {$\dots\displaystyle{\Union_{l<\delta}M^{\lambda}_l}$};
\draw (-.4,-1.15) node {$\T^{b}\in\K^*_{\lambda,\alpha}$};
\draw (-.4,-1.75) node {$\T^{*}\in\K^*_{\lambda,\alpha}$};
\draw (.8,-1.75) node {$M^{*}_0$};
\draw(1.4,-1.75) node {$M^*_1$};
\draw (1.8,-1.75) node {$ M^{*}_\beta$};
\draw (2.3,-1.75) node {$M^{*}_{\beta+1}$};
\node at (3.75,.75)[circle, fill, draw, label=90:$b$] {};
\node at (2.25,.65)[circle, fill, draw, label=290:$a_\beta$] {};
\node at (1.1,.65)[circle, fill, draw, label=290:$a_1$] {};
\draw (3.65, -.6) node {$M^{\lambda}_\delta$};
\draw (3.1,-1.8) node {$\dots\displaystyle{\Union_{\beta<\delta}M^*_\beta}=M^*_\delta$};
\end{tikzpicture}
\caption{The  towers in the proof of Theorem \ref{transfer symmetry}} \label{fig:tower}
\end{figure}
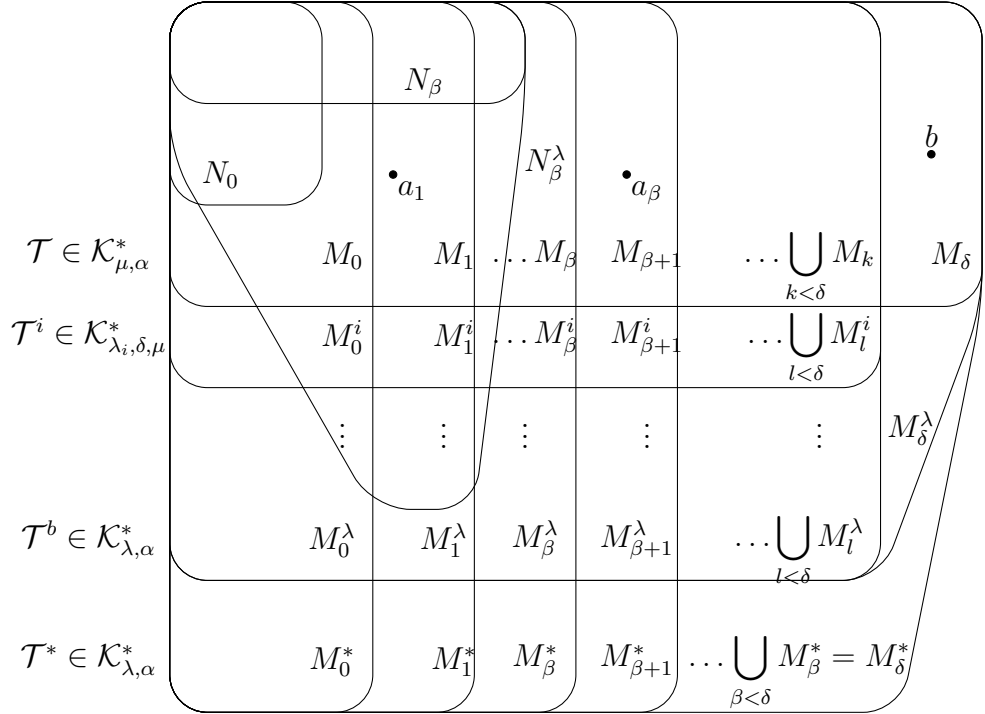

\begin{claim}\label{non-split claim}
For every $\beta<\alpha$, $\tp(a_\beta/M^*_\beta)$ does not $\mu$-split over $N_\beta$.
\end{claim}

\begin{proof}
Since $M^\lambda_\beta$ and $M^*_\beta$ are both limit models over $N^\lambda_\beta$, by $\lambda$-symmetry and Fact \ref{uq-limit}, there exists $f:M^\lambda_\beta\cong_{N^\lambda_\beta}M^*_\beta$.
Since $\T^*$ is a tower extending $\T^b$, we know $\tp(a_\beta/M^*_\beta)$ does not $\lambda$-split over $N^\lambda_\beta$.  Therefore by the definition of non-splitting, it must be the case that $\tp(f(a_\beta)/M^*_\beta)=\tp(a_\beta/M^*_\beta)$.  From this equality of types we can fix $g\in\Aut_{M^*_\beta}(\C)$ with $g(f(a_\beta))=a_\beta$.  An application of $g\circ f$ to $(\ref{mu split equation})$ yields the statement of the claim.
\end{proof}

We can now complete the proof of Theorem \ref{transfer symmetry}.   By the continuity of $\T^*$ there exists $\beta<\delta$ so that $b\in M^*_\beta$.  We can then use $\T^*$ to construct a tower $\grave\T$ in $\K^*_{\mu,\delta+1}$ extending $\T$ so that $b\in \grave M_\beta$ contradicting our assumption that $\T$ was reduced.   This is possible by the downward L\"{o}wenheim property of abstract elementary classes, $\mu$-stability, universality of the models in $\T^*$, monotonicity of non-$\mu$-splitting, and Claim \ref{non-split claim}.

\end{proof}

Similar to the proof of \cite[Theorem 2]{vandieren-symmetry-apal} we can use Lemma \ref{extension lemma} to derive symmetry from categoricity. More precisely, it is enough to assume that all the models in the top cardinal have enough saturation.

\begin{theorem}\label{categ-sym}
  Suppose $\lambda$ and $\mu$ are cardinals so that $\lambda>\mu\geq\LS(\K)$.

If $\K$ is superstable in every $\chi \in [\mu, \lambda)$, and all the models of size $\lambda$ are $\mu^+$-saturated, then $\K$ has $\mu$-symmetry.
\end{theorem}
\begin{proof}
Suppose that $\K$ does not satisfy $\mu$-symmetry.  Then by Fact \ref{sym-reduced-tower} there is a reduced discontinuous tower in $\K^*_{\mu,\alpha}$. As in the proof of Theorem \ref{transfer symmetry}, we can find a discontinuous reduced tower  $\T\in\K^*_{\mu,\alpha}$ with $\alpha=\delta+1$ with the witness of discontinuity $b\in M_\delta\backslash \Union_{\beta<\delta}M_\beta$ .  As in the proof of Theorem \ref{transfer symmetry}, we can use Lemma \ref{extension lemma} (note that we only use the first part so not assuming $\lambda$-superstability is okay) to find a tower $\T^\lambda\in\K^*_{\lambda,\mu,\delta}$ extending $\T\restriction\delta$.  

By our assumption that all the models of size $\lambda$ are $\mu^+$-saturated, $\tp(b/\Union_{\beta<\delta}M_\beta)$ is realized in $\Union_{\beta<\delta}M^\lambda_\beta$.  Let $b'$ and $\beta'<\delta$ be such that
$b'\models \tp(b/\Union_{\beta<\delta}M_\beta)$ and $b'\in M^\lambda_{\beta'}$.  Fix $f\in\Aut_{\Union_{\beta<\delta}M_\beta}(\C)$ so that $f(b')=b$.  Notice that $\T^b:=f(\T^\lambda)$ is a tower in $\K^*_{\lambda,\delta,\mu}$ extending $\T\restriction\delta$ with $b\in M^b_{\beta'}$.  

We can now use the downward L\"{o}wenheim-Skolem property of abstract elementary classes, stability in $\mu$, $\mu^+$-saturation of models of cardinality $\lambda$, and monotonicity of non-$\mu$-splitting to construct from $\T^b$ a discontinuous tower  in $\K^*_{\mu,\alpha}$ extending $\T$ so that $b$ appears in the model indexed by $\beta'$ in the tower.  This will contradict our choice of $\T$ being reduced.
\end{proof}
\begin{remark}
  Instead of assuming that all the models of size $\lambda$ are $\mu^+$-saturated, it is enough to assume the following weaker property. For any $\delta < \mu^+$ and any increasing chain $\seq{M_i : i < \delta}$ in $\K_\lambda$ of $(<\lambda, \cf(\lambda))$-limit models (i.e.\ for each $i < \delta$, there exists a resolution of $M_i$ $\seq{M_i^j \in \K_{<\lambda} : j < \cf(\lambda)}$ such that $M_i^{j + 1}$ is universal over $M_i^j$ for each $j < \cf(\lambda)$), $\bigcup_{i < \delta} M_i$ is $\mu^+$-saturated.
\end{remark}

\section{A hierarchy of symmetry properties}\label{sym-props-sec}

We discuss the relationship between the symmetry property of Definition \ref{sym defn} and other symmetry properties previously defined in the literature, especially the symmetry property in the definition of a good $\mu$-frame. This expands on the short remark after Definition 3 of \cite{vandieren-symmetry-apal} and on Corollary 2 there. It will be convenient to use the following terminology. This appears already in \cite[Definition 3.8]{ss-tame-jsl}.

\begin{definition}\label{mu-forking-def}
  Let $M_0 \lea M \lea N$ be models in $\K_\mu$. We say a type $p \in \gS (N)$ \emph{explicitly does not $\mu$-fork over $(M_0, M)$} if:

    \begin{enumerate}
      \item $M$ is universal over $M_0$.
      \item $p$ does not $\mu$-split over $M_0$.
    \end{enumerate}

    We say that \emph{$p$ does not $\mu$-fork over $M$} if there exists $M_0$ so that $p$ explicitly does not $\mu$-fork over $(M_0, M)$.
\end{definition}
\begin{remark}
  Assuming $\mu$-superstability, the relation ``$p$ does not $\mu$-fork over $M$'' is very close to defining an independence notion with the properties of forking in a first-order superstable theory (i.e.\ a good $\mu$-frame, see below). In fact using tameness it can be used to do precisely that, see \cite{ss-tame-jsl} or Theorem \ref{good-frame-weak-tameness}. Moreover forking in any categorical good $\mu$-frame has to be $\mu$-forking, see Fact \ref{canon-fact}.
\end{remark}

We now give several variations on $\mu$-symmetry. We will show that variation (\ref{many-syms-1}) is equivalent to (\ref{many-syms-2}) which implies (\ref{many-syms-3}) which implies (\ref{many-syms-4}). Moreover variation (\ref{many-syms-1}) is equivalent to the $\mu$-symmetry of Definition \ref{sym defn} and variation (\ref{many-syms-4}) is equivalent to the symmetry property of good frames. We do not know if any of the implications can be reversed, or even if all the variations already follow from superstability (see Question \ref{good-frame-sym-q}).

For clarity, we have highlighted the differences between each property. 

\begin{definition}\label{many-syms}
  Let $\mu \ge \LS (\K)$.
  \begin{enumerate}
  \item\label{many-syms-1} $\K$ has \emph{uniform $\mu$-symmetry} if for any limit models $N, M_0, M$ in $\K_\mu$ where $M$ is limit over $M_0$ and $M_0$ is limit over $N$, if \underline{$\gtp (b / M)$ does not $\mu$-split over $M_0$}, $a \in |M|$, and $\gtp (a / M_0)$ explicitly does not $\mu$-fork over $(N, M_0)$, there exists $M_b \in \K_{\mu}$ containing $b$ and limit over $M_0$ so that \textbf{$\gtp (a / M_b)$ explicitly does not $\mu$-fork over $(N, M_0)$}.
    \item\label{many-syms-2} $\K$ has \emph{weak uniform $\mu$-symmetry} if for any limit models $N, M_0, M$ in $\K_\mu$ where $M$ is limit over $M_0$ and $M_0$ is limit over $N$, if \underline{$\gtp (b / M)$ does not $\mu$-fork over $M_0$}, $a \in |M|$, and $\gtp (a / M_0)$ explicitly does not $\mu$-fork over $(N, M_0)$, there exists $M_b \in \K_{\mu}$ containing $b$ and limit over $M_0$ so that \textbf{$\gtp (a / M_b)$ explicitly does not $\mu$-fork over $(N, M_0)$}.  See Figure \ref{fig:weak uniform sym}.
    \item\label{many-syms-3} $\K$ has \emph{non-uniform $\mu$-symmetry} if for any limit models $M_0, M$ in $\K_\mu$ where $M$ is limit over $M_0$, if \underline{$\gtp (b / M)$ does not $\mu$-split over $M_0$}, $a \in |M|$, and $\gtp (a / M_0)$ does not $\mu$-fork over $M_0$, there exists $M_b \in \K_{\mu}$ containing $b$ and limit over $M_0$ so that \textbf{$\gtp (a / M_b)$ does not $\mu$-fork over $M_0$}.
    \item\label{many-syms-4} $\K$ has \emph{weak non-uniform $\mu$-symmetry} if for any limit models $M_0, M$ in $\K_\mu$ where $M$ is limit over $M_0$, if \underline{$\gtp (b / M)$ does not $\mu$-fork over $M_0$}, $a \in |M|$, and $\gtp (a / M_0)$ does not $\mu$-fork over $M_0$, there exists $M_b \in \K_{\mu}$ containing $b$ and limit over $M_0$ so that \textbf{$\gtp (a / M_b)$ does not $\mu$-fork over $M_0$}.
  \end{enumerate}
\end{definition}

\begin{figure}[h]
\begin{tikzpicture}[rounded corners=5mm, scale=3,inner sep=.5mm]
\draw (0,1.25) rectangle (.5,.5);
\draw (.4,1.1) node {$N$};
\draw (0,0) rectangle (3,1.25);
\draw (0,1.25) rectangle (1,0);
\draw (.85,.2) node {$M_0$};
\draw (3.2, .25) node {$M$};
\draw (.85, .9) node {$M'_0$};
\draw (.3,.8) rectangle (.9,.4);
\draw[color=gray] (0,1.25) rectangle (1.5, -.5);
\node at (1.1,-.25)[circle, fill, draw, label=45:$b$] {};
\node at (2,.75)[circle, fill, draw, label=45:$a$] {};
\draw[color=gray] (1.75,-.25) node {$M^{b}$};
\end{tikzpicture}
\caption{A diagram of the models and elements in the definition of weak uniform $\mu$-symmetry.  We require that $\gtp (b / M)$ does not $\mu$-\emph{fork} over $M_0$ in the weak version, so there exists $M_0'$ such that $M_0$ is limit over $M_0'$ and $\gtp (b / M)$ does not $\mu$-split over $M_0'$} \label{fig:weak uniform sym}
\end{figure}
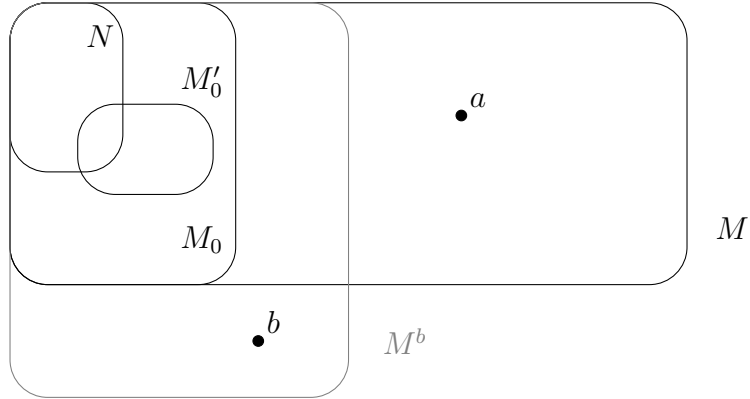

The difference between the uniform and non-uniform variations is in the conclusion: in the uniform case, we start with $\gtp (a / M_0)$ which explicitly does not $\mu$-fork over $(N, M_0)$ and get $\gtp (a / M_b)$ explicitly does not $\mu$-fork over $(N, M_0)$. Thus both types do not $\mu$-split over $N$. In the non-uniform case, we start with $\gtp (a / M_0)$ which does not $\mu$-fork over $M_0$, hence explicitly does not $\mu$-fork over $(N, M_0)$ for some $N$, but we only get that $\gtp (a / M_b)$ does not $\mu$-fork over $M_0$, so it explicitly does not $\mu$-fork over $(N', M_0)$, for some $N'$ potentially different from $N$. 

The difference between weak and non-weak is in the starting assumption: in the weak case, we assume that $\gtp (b / M)$ does not $\mu$-fork over $M_0$, hence there exists $M_0'$ so that $M_0$ is limit over $M_0'$ and $\gtp (b / M)$ does not $\mu$-split over $M_0'$. In the non-weak case, we assume only that $\gtp (b / M)$ does not $\mu$-\emph{split} over $M_0$. Even under $\mu$-superstability, it is open whether this implies that there must exist a smaller $M_0'$ so that $\gtp (b / M)$ does not $\mu$-split over $M_0'$. The problem is that $\mu$-splitting need not satisfy the transitivity property, see the discussion after Definition 3.8 in \cite{ss-tame-jsl}.

Using the monotonicity property of $\mu$-splitting, we get the easy implications:

\begin{proposition}\label{easy-implications}
  Let $\mu \ge \LS (\K)$. If $\K$ has uniform $\mu$-symmetry, then it has non-uniform $\mu$-symmetry and weak uniform $\mu$-symmetry. If $\K$ has non-uniform $\mu$-symmetry, then it has weak non-uniform $\mu$-symmetry.
\end{proposition}

Playing with the definitions and monotonicity of $\mu$-splitting (noting that cases ruled out by Definition \ref{sym defn} such as $a \in |M_0|$ are easy to handle), we also have:

\begin{proposition}\label{unif-usual}
  $\K$ has uniform $\mu$-symmetry if and only if it has $\mu$-symmetry (in the sense of Definition \ref{sym defn}).
\end{proposition}

Surprisingly, uniform symmetry and weak uniform symmetry are also equivalent assuming superstability. We will use the characterization of symmetry in terms of reduced towers provided by Fact \ref{sym-reduced-tower}.

\begin{lemma}\label{weak-unif-equiv}
  If $\K$ is $\mu$-superstable, then weak uniform $\mu$-symmetry is equivalent to uniform $\mu$-symmetry.
\end{lemma}
\begin{proof}
  By Proposition \ref{easy-implications}, uniform symmetry implies weak uniform symmetry. Now assuming weak uniform symmetry, the proof of $(\ref{symmetry item})\Rightarrow(\ref{reduced are continuous})$ of Fact \ref{sym-reduced-tower} still goes through. The point is that whenever we consider $\gtp (b / M)$ in the proof, $M = \bigcup_{i < \delta} M_i$ for some increasing continuous $\seq{M_i : i < \delta}$ with $M_{i + 1}$ universal over $M_i$ for all $i < \delta$, and we simply use that by superstability $\gtp (b / M)$ does not $\mu$-split over $M_i$ for some $i < \delta$. However we also have that $\gtp (b / M)$ explicitly does not $\mu$-fork over $(M_i, M_{i + 1})$.

  Therefore reduced towers are continuous, and hence by Fact \ref{sym-reduced-tower} $\K$ has $\mu$-symmetry (and so by Proposition \ref{unif-usual} uniform $\mu$-symmetry).
\end{proof}

How do these definitions compare to the symmetry property in good $\mu$-frames? Recall \cite[Definition II.2.1]{shelahaecbook} that a good $\mu$-frame is a triple $\s = (\K_\mu, \nf, \Sbs)$ where:

\begin{enumerate}
  \item $\K$ is an AEC.
  \item For each $M \in \K_\mu$, $\Sbs (M)$ (called the set of \emph{basic types} over $M$) is a set of nonalgebraic Galois types over $M$ satisfying (among others) the \emph{density property}: if $M \lta N$ are in $\K_\mu$, there exists $a \in |N| \backslash |M|$ such that $\gtp (a / M; N) \in \Sbs (M)$.
  \item $\nf$ is an (abstract) independence relation on types of length one over models in $\K_\mu$ satisfying several basic properties (that we will not list here) of first-order forking in a superstable theory.
\end{enumerate}

\begin{remark}
We will \emph{not} use the axiom (B) \cite[Definition II.2.1]{shelahaecbook} requiring the existence of a superlimit model of size $\mu$. In fact many papers (e.g.\ \cite{jrsh875}) define good frames without this assumption.
\end{remark}

As in \cite[Definition II.6.35]{shelahaecbook}, we say that a good $\mu$-frame $\s$ is \emph{type-full} if for each $M \in \K_\mu$, $\Sbs (M)$ consists of \emph{all} the nonalgebraic types over $M$. For simplicity, we focus on type-full good frames in this paper. Given a type-full good $\mu$-frame $\s = (\K_\mu, \nf, \Sbs)$ and $M_0 \lea M$ both in $\K_\mu$, we say that a nonalgebraic type $p \in \gS (M)$ \emph{does not $\s$-fork over $M_0$} if it does not fork over $M_0$ according to the abstract independence relation $\nf$ of $\s$. We say that a good $\mu$-frame $\s$ is \emph{on $\K_\mu$} if its underlying class is $\K_\mu$.

The existence of a good $\mu$-frame gives quite a lot of information about the class. 

\begin{fact}\label{good-frame-facts}
  Assume there is a good $\mu$-frame on $\Ksatp{\lambda}_\mu$, for $\lambda \le \mu$ (so in particular, unions of chains of $\lambda$-saturated models are $\lambda$-saturated). Then:

  \begin{enumerate}
    \item For any $M_0, M_1, M_2 \in \K_\mu$ such that $M_1$ and $M_2$ are limit over $M_0$, $M_1 \cong_{M_0} M_2$. 
    \item $\K$ is $\mu$-superstable.
  \end{enumerate}
\end{fact}
\begin{proof}
  The first part is \cite[Lemma II.4.8]{shelahaecbook} (or see \cite[Theorem 9.2]{ext-frame-jml}). Note that by the usual back and forth argument (as made explicit in the proof of \cite[Theorem 7.1]{bv-sat-v3}), any limit model is isomorphic to a limit model where the models witnessing it are $\lambda$-saturated. The second part is because:
  \begin{itemize}
    \item By definition of a good $\mu$-frame, $\mu \ge \LS (\K)$, $\K_\mu$ is nonempty, has amalgamation, joint embedding, and no maximal models.
    \item By \cite[Claim II.4.2.(1)]{shelahaecbook}, $\K$ is stable in $\mu$.
    \item By the uniqueness property of $\s$-forking, if a type does not $\s$-fork over $M_0$ (where $\s$ is a good $\mu$-frame on $\K_\mu$), then it does not $\mu$-split over $M_0$ (see \cite[Lemma 4.2]{bgkv-apal}). Thus we obtain the ``no long splitting chains'' condition in Definition \ref{ss assm} when the $M_i$s are $\lambda$-saturated, and as noted above (or in \cite[Proposition 10.6]{indep-aec-apal}) we can do a back and forth argument to get that no long splitting chains holds even when the members of the chain are not saturated.
  \end{itemize}
\end{proof}

Among the axioms a good $\mu$-frame must satisfy is the symmetry axiom:

\begin{definition}\label{good-frame-sym}
  The \emph{symmetry axiom} for a good $\mu$-frame $\s = (\K_\mu, \nf, \Sbs)$ is the following statement: For any $M_0 \lea M$ in $\K_\mu$, if $\gtp (b / M)$ does not $\s$-fork over $M_0$ and $a \in |M| \backslash |M_0|$ is so that $\gtp (a / M_0) \in \Sbs (M_0)$, there exists $M_b \in \K_\mu$ containing $b$ and extending $M_0$ so that $\gtp (a / M_b)$ does not $\s$-fork over $M_0$.
  
Note that the good frame axioms imply that $\K$ has amalgamation in $\mu$, so for this definition (and for simplicity only) we work inside a saturated model $\sea$ of size $\mu^+$.
\end{definition}

Since the symmetry properties of Definition \ref{many-syms} are all over limit models only, we will discuss only frames whose models are the limit models. By Fact \ref{good-frame-facts}, such frames are categorical (that is, their underlying class has a single model up to isomorphism). This is not a big loss since most known general constructions of a good $\mu$-frame (e.g.\ \cite[Theorem II.3.7]{shelahaecbook}, \cite[Theorem 1.3]{ss-tame-jsl}) assume categoricity in $\mu$. In the known constructions when categoricity in $\mu$ is not assumed (such as in \cite[Corollary 10.19]{indep-aec-apal}), it holds that the union of a chain of $\mu$-saturated model is $\mu$-saturated, so we can simply restrict the frame to the saturated models of size $\mu$.

We will use \cite[Theorem 9.7]{indep-aec-apal} that categorical good $\mu$-frames are canonical:

\begin{fact}[The canonicity theorem for categorical good frames]\label{canon-fact}
  Let $\s = (\K_\mu, \nf, \Sbs)$ be a categorical good $\mu$-frame. Let $p \in \Sbs (M)$ and let $M_0 \lea M$ be in $\K_\mu$. Then $p$ does not $\s$-fork over $M_0$ if and only if $p$ does not $\mu$-fork over $M_0$ (recall Definition \ref{mu-forking-def}).
\end{fact}
\begin{remark}\label{symmetry-rmk}
  The proof of the second part of Fact \ref{good-frame-facts} and Fact \ref{canon-fact} do not use the symmetry axiom (but the first part of Fact \ref{good-frame-facts} does). 
\end{remark}

Using the canonicity theorem, we obtain:

\begin{theorem}\label{sym-good-frame-equiv}
  Let $\s$ be a type-full categorical good $\mu$-frame on $\K_\mu$, except that we do not assume that it satisfies the symmetry axiom. The following are equivalent:

  \begin{enumerate}
    \item $\s$ satisfies the symmetry axiom (Definition \ref{good-frame-sym}).
    \item $\K$ has weak non-uniform $\mu$-symmetry (Definition \ref{many-syms}.(\ref{many-syms-4})).
  \end{enumerate}
\end{theorem}
\begin{proof}
  By Fact \ref{canon-fact} (and Remark \ref{symmetry-rmk}), $\mu$-forking and $\s$-forking coincide. Now replace $\s$-forking by $\mu$-forking in the symmetry axiom and expand the definition.
\end{proof}

One can ask whether weak non-uniform symmetry can be replaced by the uniform version:

\begin{question}\label{good-frame-sym-q}
  Assume there is a type-full categorical good $\mu$-frame on $\K_\mu$. Does $\K$ have $\mu$-symmetry? More generally, if $\K$ is $\mu$-superstable, does $\K$ have $\mu$-symmetry?
\end{question}

We will show (Corollary \ref{mu-symmetry-1}) that the answer is positive if $\K$ is $\mu$-tame. Still much less suffices: 

\begin{theorem}\label{sym-good-frame}
  If $\K$ is $\mu$-superstable and has a good $\mu^+$-frame on $\Ksatp{\lambda}_{\mu^+}$ for some $\lambda \le \mu^+$, then $\K$ has $\mu$-symmetry.
\end{theorem}
\begin{proof}
  By Fact \ref{good-frame-facts}, all limit models in $\K_{\mu^+}$ are saturated and $\K$ is $\mu^+$-superstable. By Fact \ref{union-sat-monica}, the remark following it, and Fact \ref{chainsat-sym}, $\K$ has $\mu$-symmetry.
\end{proof}

We end this section with a partial answer to Question \ref{good-frame-sym-q} assuming that the good frame satisfies several additional technical properties of frames introduced by Shelah (see \cite[Definitions III.1.1, III.1.3]{shelahaecbook}). For this result amalgamation (Hypothesis \ref{ap-hyp}) is not necessary.

\begin{corollary}
  Assume there is a successful $\text{good}^+$ $\mu$-frame with underlying class $\K_\mu$. Then $\K$ has $\mu$-symmetry.
\end{corollary}
\begin{proof}
  Let $\s$ be a successful $\text{good}^+$ $\mu$-frame with underlying class $\K_\mu$. By \cite[Claim II.6.36]{shelahaecbook}, we can assume without loss of generality that $\s$ is type-full (note that by \cite[Theorem 6.13]{bgkv-apal} there can be only one such type-full frame). By Fact \ref{good-frame-facts}, $\K$ is $\mu$-superstable.  By \cite[III.1.6, III.1.7, III.1.8]{shelahaecbook}, there is a good $\mu^+$-frame $\s^+$ on $\Ksatp{\mu^+}_{\mu^+}$. By Theorem \ref{sym-good-frame}, $\K$ has $\mu$-symmetry.
\end{proof}


\section{Symmetry from no order property}\label{sym-op-sec}

In this section, we give another way to derive symmetry. The idea is to imitate the argument from \cite[Theorem 5.14]{bgkv-apal}, but we first have to obtain enough properties of independence. We will work with $\mu$-forking (Definition \ref{mu-forking-def}). We start by improving Proposition \ref{extension for non-splitting}.

\begin{proposition}[Extension property of forking]\label{mu-forking-props}
  Let $\LS (\K) \le \mu \le \lambda$. Let $M \lea N$ be in $K_{[\mu,\lambda]}$. Let $p \in \gS (M)$ be such that $p$ explicitly does not $\mu$-fork over $(M_0, M)$. If $\K$ is superstable in every $\chi \in [\mu, \lambda]$, then there exists $q \in \gS (N)$ extending $p$ and explicitly not $\mu$-forking over $(M_0, M)$. Moreover $q$ is algebraic if and only if $p$ is.
\end{proposition}
\begin{proof}
  By induction on $\|N\|$. Let $a$ realize $p$. If $\|N\| = \|M\|$ this is given by Proposition \ref{extension for non-splitting} (if $\|M\| = \|N\| = \mu$, this is \cite[Theorem I.4.10]{vandierennomax}). If $\|M\| < \|N\|$, build $\seq{N_i \in \K_{\|M\| + |i|} : i \le \|N\|}$ increasing continuous such that $N_0 = M$, $N_{i + 1}$ is limit over $N_i$, and $\gtp (a / N_i)$ explicitly does not $\mu$-fork over $(M_0, M)$. This is possible by the induction hypothesis and the continuity property of splitting (Proposition \ref{limit splitting proposition}). Now $N_\lambda$ is $\|N\|$-universal over $N_0 = M$, so let $f: N \xrightarrow[M]{} N_\lambda$. Let $q := f^{-1} (\gtp (a / f[N]))$. It is easy to check that $q$ is as desired.
\end{proof}

Next we recall the definition of the order property in AECs \cite[Definition 4.3]{sh394}.

\begin{definition}
  Let $\alpha$ and $\lambda$ be cardinals. A model $M \in \K$ has the \emph{$\alpha$-order property of length $\lambda$} if there exists $\seq{\ba_i : i < \lambda}$ inside $M$ with $\ell (\ba_i) = \alpha$ for all $i < \lambda$, such that for any $i_0 < j_0 < \lambda$ and $i_1 < j_1 < \lambda$, $\gtp (\ba_{i_0} \ba_{j_0} / \emptyset) \neq \gtp (\ba_{j_1} \ba_{i_1} / \emptyset)$. 

  We say that \emph{$\K$ has the $\alpha$-order property of length $\lambda$} if some $M \in \K$ has it. We say that \emph{$\K$ has the $\alpha$-order property} if it has the $\alpha$-order property of length $\lambda$ for all cardinals $\lambda$.
\end{definition}

We will use two important facts.  The first says that it is enough to look at length up to the Hanf number. The second is that the order property implies instability.

\begin{definition}\label{hanf-def}
  For $\lambda$ an infinite cardinal, $\hanf{\lambda} := \beth_{(2^{\lambda})^+}$.
\end{definition}

\begin{fact}[Claim 4.5.3 in \cite{sh394}]\label{shelah-hanf}
  Let $\alpha$ be a cardinal. If $\K$ has the $\alpha$-order property of length $\lambda$ for all $\lambda < \hanf{\alpha + \LS (K)}$, then $\K$ has the $\alpha$-order property.
\end{fact}
\begin{fact}\label{stab-facts-op}
  If $\K$ has the $\alpha$-order property and $\mu \ge \LS (\K)$ is such that $\mu = \mu^{\alpha}$, then $\K$ is not stable in $\mu$.
\end{fact}
\begin{proof}
  By \cite[Claim 4.8.2]{sh394} (see \cite[Fact 5.13]{bgkv-apal} for a proof), there exists $M \in \K_\mu$ such that $|\gS^{\alpha} (M)| > \mu$. By \cite[Theorem 3.1]{longtypes-ndjfml}, $\K$ is not stable in $\mu$.
\end{proof}

The following lemma appears in some more abstract form in \cite[Lemma 5.6]{bgkv-apal}.  The lemma says that if we assume that $p$ does not $\mu$-fork over $M$,
then in the definition of non-splitting (Definition \ref{def:splitting}) we can  replace the $N_\ell$ by arbitrary sequences in $N$ of length at most $\mu$. In the proof of Lemma \ref{sym-lem}, this will be used for sequences of length one.

\begin{lemma}\label{ns-lemma}
  Let $\mu \ge \LS (\K)$. Let $M \in \K_\mu$ and $N \in \K_{\ge \mu}$ be such that $M \lea N$. Assume that $\K$ is stable in $\mu$. If $p \in \gS (N)$ does not $\mu$-fork over $M$ (Definition \ref{mu-forking-def}), $a$ realizes $p$, and $\bb_1, \bb_2 \in \fct{\le \mu}{|N|}$ are such that $\gtp (\bb_1 / M) = \gtp (\bb_2 / M)$, then $\gtp (a \bb_1 / M) = \gtp (a \bb_2 / M)$.
\end{lemma}

\begin{proof}
Pick $N_0 \in \K_\mu$ containing $\bb_1 \bb_2$ with $M \lea N_0 \lea N$. Then $p \rest N_0$ does not $\mu$-fork over $M$.   Replacing $N$ by $N_0$ if necessary, we can assume without loss of generality that $N \in \K_\mu$. By definition of $\mu$-forking, there exists $M_0 \in \K_\mu$ such that $M_0 \lea M$ and $p$ does not $\mu$-split over $M_0$. By the extension and uniqueness property for $\mu$-splitting there exists $N'$ extending $N$ of cardinality $\mu$ so that $N'$ is universal over both $N$ and $M$, and $\gtp(a/N')$ does not $\mu$-split over $M_0$.  Since  $\gtp (\bb_1 / M) = \gtp (\bb_2 / M)$ and since $N'$ is universal over $N$, we can find $f:N\underset{M}\rightarrow N'$ so that $f(\bb_1)=\bb_2$.
Since $\gtp(a/N')$ does not $\mu$-split over $M_0$ we know $\gtp(f(a)/f(N))=\gtp(a/f(N))$.  By our choice of $f$ this implies that there exists $g\in\Aut_{f(N)}(\C)$ so that 
$g(f(a))=a$, $g\restriction M=\id_M$, and $g(\bb_2)=\bb_2$.  In other words $\gtp(f(a)\bb_2/M)=\gtp(a\bb_2/M)$.
Moreover $f^{-1}$ witnesses that $\gtp (a \bb_1 / M) = \gtp (f(a) \bb_2 / M)$, which we have seen is equal to $\gtp(a\bb_2/M)$.
\end{proof}

The next lemma shows that failure of symmetry implies the order property. The proof is similar to that of \cite[Theorem 5.14]{bgkv-apal}, the difference is that we use Lemma \ref{ns-lemma} and the equivalence between symmetry and weak uniform symmetry (Lemma \ref{weak-unif-equiv}).

\begin{lemma}\label{sym-lem}
  Let $\lambda > \mu \ge \LS (\K)$. Assume that $\K$ is superstable in every $\chi \in [\mu, \lambda)$. If $\K$ does \emph{not} have $\mu$-symmetry, then it has the $\mu$-order property of length $\lambda$.
\end{lemma}
\begin{proof}
  By Lemma \ref{weak-unif-equiv}, $\K$ does not have weak uniform $\mu$-symmetry. We first pick witnesses to that fact. Pick limit models $N, M_0, M \in \K_\mu$ such that $M$ is limit over $M_0$ and $M_0$ is limit over $N$. Pick $b$ such that $\gtp (b / M)$ does not $\mu$-fork over $M_0$, $a \in |M|$, and $\gtp (a / M_0)$ explicitly does not $\mu$-fork over $(N, M_0)$, and there does \emph{not} exist $M_b \in \K_{\mu}$ containing $b$ and limit over $M_0$ so that $\gtp (a / M_b)$ explicitly does not $\mu$-fork over $(N, M_0)$. We will show that $\sea$ has the $\mu$-order property of length $\lambda$.

  We build increasing continuous $\seq{N_\alpha : \alpha < \lambda}$ and $\seq{a_\alpha, b_\alpha, N_\alpha' : \alpha < \lambda}$ by induction so that for all $\alpha < \lambda$:

  \begin{enumerate}
  \item $N_\alpha, N_\alpha' \in \K_{\mu + |\alpha|}$.
  \item $N_0$ is limit over $M$ and $b \in |N_0|$.
  \item $\gtp (a_\alpha / M_0) = \gtp (a / M_0)$ and $a_\alpha \in |N_\alpha'|$.
  \item\label{req-4} $\gtp (b_\alpha / M) = \gtp (b / M)$ and $b_\alpha \in |N_{\alpha + 1}|$.
  \item $N_\alpha'$ is limit over $N_\alpha$ and $N_{\alpha + 1}$ is limit over $N_{\alpha}'$. 
  \item $\gtp (a_\alpha / N_\alpha)$ explicitly does not $\mu$-fork over $(N, M_0)$ and $\gtp (b_\alpha / N_\alpha')$ does not $\mu$-fork over $M_0$.
\end{enumerate}

\paragraph{\underline{This is possible}}

Let $N_0$ be any model in $\K_\mu$ containing $b$ that is limit over $M$. At $\alpha$ limits, let $N_\alpha := \bigcup_{\beta < \alpha} N_\beta$. Now assume inductively that $N_\beta$ has been defined for $\beta \le \alpha$, and $a_\beta$, $b_\beta$, $N_\beta'$ have been defined for $\beta < \alpha$. By extension for splitting, find $q \in \gS (N_\alpha)$ that explicitly does not $\mu$-fork over $(N, M_0)$ and extends $\gtp (a / M_0)$. Let $a_\alpha$ realize $q$ and pick $N_\alpha'$ limit over $N_\alpha$ containing $a_\alpha$. Now by extension again, find $q' \in \gS (N_\alpha')$ that does not $\mu$-fork over $M_0$ and extends $\gtp (b / M)$. Let $b_\alpha$ realize $q'$ and pick $N_{\alpha + 1}$ limit over $N_\alpha'$ containing $b_\alpha$.

\paragraph{\underline{This is enough}}

We show that for $\alpha, \beta < \lambda$:

\begin{enumerate}
\item\label{cond-0} $\gtp (a_\alpha b / M_0) \neq \gtp (a b / M_0)$
\item\label{cond-01} $\gtp (a b_\beta / M_0) = \gtp (a b / M_0)$
  \item\label{cond-1} If $\beta < \alpha$, $\gtp (ab / M_0) \neq \gtp (a_\alpha b_\beta / M_0)$.
  \item\label{cond-2} If $\beta \ge \alpha$, $\gtp (ab / M_0) = \gtp (a_\alpha b_\beta / M_0)$.
\end{enumerate}

For (\ref{cond-0}), observe that $b \in |N_0| \subseteq |N_\alpha|$ and $\gtp (a_\alpha / N_\alpha)$ explicitly does not $\mu$-fork over $(N, M_0)$. Therefore by monotonicity $N_\alpha$ witnesses that there exists $N_b \in \K_\mu$ containing $b$ and limit over $M_0$ so that $\gtp (a_\alpha / N_b)$ explicitly does not $\mu$-fork over $(N, M_0)$. By failure of symmetry and invariance, we must have that $\gtp (a_\alpha b / M_0) \neq \gtp (a b / M_0)$.

For (\ref{cond-01}), use the assumption that $a \in |M|$ together with clause (\ref{req-4}) of the construction.

For (\ref{cond-1}), suppose $\beta < \alpha$.  We know that $\gtp (a_\alpha / N_\alpha)$ does not $\mu$-fork over $M_0$. Since $\beta < \alpha$, $b, b_\beta \in |N_\alpha|$ and $\gtp (b / M) = \gtp (b_\beta / M)$, we must have by Lemma \ref{ns-lemma} that $\gtp (a_\alpha b / M_0) = \gtp (a_\alpha b_\beta / M_0)$.  Together with (\ref{cond-0}), this implies $\gtp (a b / M_0) \neq \gtp (a_\alpha b_\beta / M_0)$.  This is really where we use the equivalence between uniform $\mu$-symmetry and weak uniform $\mu$-symmetry: if we only had failure of uniform $\mu$-symmetry, then we would only know that $\gtp (b / M)$ does not \emph{$\mu$-split} over $M_0$, so would be unable to use Lemma \ref{ns-lemma}. 

To see (\ref{cond-2}), suppose $\beta \geq \alpha$ and recall that (by (\ref{cond-01})) $\gtp (a b / M_0) = \gtp (a b_\beta / M_0)$. We also have that $\gtp (b_\beta / N_\beta')$ does not $\mu$-fork over $M_0$. Moreover $\gtp (a / M_0) = \gtp (a_\alpha / M_0)$, and $a, a_\alpha \in N_{\beta}'$. By Lemma \ref{ns-lemma} again, $\gtp (a b_\beta / M_0) = \gtp (a_\alpha b_\beta / M_0)$. This gives us that $\gtp (ab / M_0) = \gtp (a_\alpha b_\beta / M_0)$.

Now let $\bar{d}$ be an enumeration of $M_0$ and for $\alpha < \lambda$, let $\bc_\alpha := a_\alpha b_\alpha \bar{d}$. Then (\ref{cond-1}) and (\ref{cond-2}) together tell us that the sequence $\seq{\bc_\alpha \mid \alpha < \lambda}$ witnesses the $\mu$-order property of length $\lambda$.
\end{proof}

We conclude that symmetry follows from enough instances of superstability.

\begin{theorem}\label{sym-from-superstab}
  Let $\mu \ge \LS (\K)$. Then there exists $\lambda < \hanf{\mu}$ such that if $\K$ is superstable in every $\chi \in [\mu, \lambda)$, then $\K$ has $\mu$-symmetry.
\end{theorem}
\begin{proof}
  If $\K$ is unstable in $2^{\mu}$, then we can set $\lambda := \left(2^{\mu}\right)^+$ and get a vacuously true statement; so assume that $\K$ is stable in $2^{\mu}$. By Fact \ref{stab-facts-op}, $\K$ does not have the $\mu$-order property. By Fact \ref{shelah-hanf}, there exists $\lambda < \hanf{\mu}$ such that $\K$ does not have the $\mu$-order property of length $\lambda$. By Lemma \ref{sym-lem}, it is as desired.
\end{proof}

\begin{remark}
  How can one obtain many instances of superstability as in the hypothesis of Theorem \ref{sym-from-superstab}? One way is categoricity, see Fact \ref{shelah-villaveces}. Another way is to start with one instance of superstability and transfer it up using tameness, see Fact \ref{ss-transfer}.
\end{remark}

\section{Symmetry and tameness}\label{sym-tame-sec}

Tameness is a locality property for types introduced by Grossberg and VanDieren in \cite{tamenessone} and used to prove Shelah's eventual categoricity conjecture from a successor in \cite{tamenesstwo}. It has also played a key roles in the proof of several other categoricity transfers, for example \cite{tamelc-jsl, ap-universal-v10, categ-universal-2-v3-toappear}. 

\begin{definition}[Tameness]\label{tameness-def}
  Let $\mu \ge \LS (\K)$. $\K$ is \emph{$\mu$-tame} if for every $M \in \K$ and every $p, q \in \gS (M)$, if $p \neq q$, then there exists $M_0 \in \K_{\le \mu}$ with $M_0 \lea M$ such that $p \rest M_0 \neq q \rest M_0$.
\end{definition}

In this section, we study the combination of tameness (and its relatives, see below) with superstability. In Section \ref{sym-categ-sec}, we will combine tameness and categoricity.

\subsection{Weak tameness}

We will start by studying a weaker, more local, variation that appears already in \cite{sh394}. We use the notation in \cite[Definition 11.6]{baldwinbook09}.

\begin{definition}[Weak tameness]
  Let $\chi, \mu$ be cardinals with $\LS (\K) \le \chi \le \mu$. $\K$ is \emph{$(\chi, \mu)$-weakly tame} if for any \emph{saturated} $M \in \K_\mu$, any $p, q \in \gS (M)$, if $p \neq q$, there exists $M_0 \in \K_{\chi}$ with $M_0 \lea M$ and $p \rest M_0 \neq q \rest M_0$.
\end{definition}

Tameness says that types over \emph{any} models are determined by their small restrictions. Weak tameness says that only types over \emph{saturated} models have this property.

While there is no known example of an AEC that is weakly tame but not tame, it is known that weak tameness follows from categoricity in a suitable cardinal (but the corresponding result for non-weak tameness is open, see \cite[Conjecture 1.5]{tamenessthree}): this appears as \cite[Main Claim II.2.3]{sh394} and a simplified argument is in \cite[Theorem 11.15]{baldwinbook09}.

\begin{fact}\label{weak-tameness-from-categ-fact}
  Let $\lambda > \mu \ge \hanf{\LS (\K)}$. Assume that $\K$ is categorical in $\lambda$, and the model of cardinality $\lambda$ is $\mu^+$-saturated. Then there exists $\chi < \hanf{\LS (\K)}$ such that $\K$ is $(\chi, \mu)$-weakly tame.
\end{fact}

It was shown in \cite{ss-tame-jsl} (and further improvements in \cite[Section 10]{indep-aec-apal} and \cite{bv-sat-v3}) that tameness can be combined with superstability to build a good frame at a high-enough cardinal. At a meeting in the winter of 2015 in San Antonio, the first author asked whether weak tameness could be used instead. This is not a generalization for the sake of generalization because weak tameness (but not tameness) is known to follow from categoricity. We can answer in the affirmative:

\begin{theorem}\label{good-frame-weak-tameness}
  Let $\lambda > \mu \ge \LS (\K)$. Assume that $\K$ is superstable in every $\chi \in [\mu, \lambda]$ and has $\lambda$-symmetry.
  
  If $\K$ is $(\mu, \lambda)$-weakly tame, then there exists a type-full good $\lambda$-frame with underlying class $\Ksatp{\lambda}_\lambda$.
\end{theorem}
\begin{proof}
  First observe that limit models in $\K_\lambda$ are unique (by Fact \ref{uq-limit}), hence saturated. By Theorem \ref{transfer symmetry}, $\K$ has $\chi$-symmetry for every $\chi \in [\mu, \lambda]$. By Fact \ref{union-sat-monica}, for every $\chi \in [\mu, \lambda)$, $\Ksatp{\chi^+}$ (the class of $\chi^+$-saturated models in $\K_{\ge \chi^+}$) is an AEC with $\LS (\Ksatp{\chi^+}) = \chi^+$. Therefore by Lemma \ref{chainsat-lim} $\Ksatp{\lambda}$ is an AEC with $\LS (\Ksatp{\lambda}) = \lambda$. By the $\lambda$-superstability assumption, $\Ksatp{\lambda}_\lambda$ is nonempty, has amalgamation, no maximal models, and joint embedding. It is also stable in $\lambda$. We want to define a type-full good $\lambda$-frame $\s$ on $\Ksatp{\lambda}_\lambda$. We define forking in the sense of $\s$ ($\s$-forking) as follows: For $M \lea N$ saturated of size $\lambda$, a non-algebraic $p \in \gS (N)$ does not $\s$-fork over $M$ if and only if there exists $M_0 \in \K_\mu$ such that for all $N_0 \in \K_{\mu}$, if $M_0 \lea N_0 \lea N$, then $p \rest N_0$ does not $\mu$-fork over $M_0$. 

    Now most of the axioms of good frames are verified in Section 4 of \cite{ss-tame-jsl}, the only properties that remain to be checked are extension, uniqueness, and symmetry. Extension is by Proposition \ref{mu-forking-props}, and uniqueness is by uniqueness of splitting in $\mu$ (\cite[I.4.12]{vandierennomax}) and the weak tameness assumption. As for symmetry, we know that $\lambda$-symmetry holds, hence by Proposition \ref{easy-implications}, Proposition \ref{unif-usual}, and Theorem \ref{sym-good-frame-equiv} the symmetry property of good frame follows.
\end{proof}
\begin{remark}
  If $\lambda = \mu^+$ above, then the hypotheses reduce to ``$\K$ is superstable in $\mu$ and $\mu^+$ and $\K$ has $\mu^+$-symmetry''.
\end{remark}

We can combine this construction with the results of Section \ref{sym-op-sec}:

\begin{corollary}\label{good-frame-weak-tameness-2}
  Let $\lambda > \mu \ge \LS (\K)$. Assume that $\K$ is superstable in every $\chi \in [\mu, \hanf{\lambda})$. If $\K$ is $(\mu, \lambda)$-weakly tame, then there exists a type-full good $\lambda$-frame with underlying class $\Ksatp{\lambda}_{\lambda}$.
\end{corollary}
\begin{proof}
  Combine Theorem \ref{good-frame-weak-tameness} and Theorem \ref{sym-from-superstab}.
\end{proof}

\subsection{Global tameness}

For the rest of this section, we will work with global non-weak tameness. Superstability has been studied together with amalgamation and tameness in works from the second author \cite{ss-tame-jsl, indep-aec-apal, bv-sat-v3, gv-superstability-v5-toappear}. We will use the following upward transfer of superstability:

\begin{fact}[Proposition 10.10 in \cite{indep-aec-apal}]\label{ss-transfer}
  Assume $\K$ is $\mu$-superstable and $\mu$-tame. Then for all $\mu' \ge \mu$, $\K$ is $\mu'$-superstable. In particular, $\K_{\ge \mu}$ has no maximal models and is stable in all cardinals.
\end{fact}

Recall from Section \ref{sym-props-sec} that $\Ksatp{\lambda}$ denotes the class of $\lambda$-saturated models in $\K_{\ge \lambda}$. We would like to give conditions under which $\Ksatp{\lambda}$ is an AEC -- in particular unions of chains of $\lambda$-saturated models are $\lambda$-saturated. From superstability and tameness, it is known that one eventually obtains this behavior:

\begin{fact}[Theorem 7.1 in \cite{bv-sat-v3}]\label{bv-sat-fact}
  Assume $\K$ is $\mu$-superstable and $\mu$-tame. Then there exists $\lambda_0 < \beth_{(2^{\mu^+})^+}$ such that for any $\lambda \ge \lambda_0$, $\Ksatp{\lambda}$ is an AEC with $\LS (\Ksatp{\lambda}) = \lambda$.
\end{fact}

We can use this to show that superstability implies symmetry in tame AECs (obtaining another partial answer to Question \ref{good-frame-sym-q}). We also give another, more self-contained proof that does not rely on Fact \ref{bv-sat-fact}.

\begin{corollary}\label{mu-symmetry-1}
  If $\K$ is $\mu$-superstable and $\mu$-tame, then $\K$ has $\mu$-symmetry.
\end{corollary}
\begin{proof}[First proof]
  First observe that by Fact \ref{ss-transfer}, $\K$ is superstable in every $\mu' \ge \mu$. By Fact \ref{bv-sat-fact}, there exists $\lambda_0 \ge \mu$ such that $\Ksatp{\lambda_0^+}$ is an AEC. Therefore the hypotheses of Fact \ref{chainsat-sym} are satisfied, so $\K$ has $\lambda_0$-symmetry. By Theorem \ref{transfer symmetry}, $\K$ has $\mu$-symmetry.
\end{proof}
\begin{proof}[Second proof]
  As in the first proof, $\K$ is superstable in every $\mu' \ge \mu$. By Theorem \ref{sym-from-superstab}, $\K$ has $\mu$-symmetry.
\end{proof}

Thus we obtain an improvement on the Hanf number of Fact \ref{bv-sat-fact}:

\begin{corollary}\label{chain-sat-1}
  Assume $\K$ is $\mu$-superstable and $\mu$-tame. For every $\lambda > \mu$, $\Ksatp{\lambda}$ is an AEC with $\LS (\Ksatp{\lambda}) = \lambda$.
\end{corollary}
\begin{proof}
  By Fact \ref{ss-transfer} and Corollary \ref{mu-symmetry-1}, $\K$ is $\lambda$-superstable and has $\lambda$-symmetry for any $\lambda > \mu$. By Fact \ref{union-sat-monica}, $\Ksatp{\mu^+}$ is an AEC with $\LS (\Ksatp{\mu^+}) = \mu^+$. We can replace $\mu^+$ with any successor $\lambda > \mu$. To take care of limit cardinals $\lambda$, use Lemma \ref{chainsat-lim}.
\end{proof}

Note that Corollary \ref{chain-sat-1} is an improvement on Fact \ref{bv-sat-fact} and the second proof of Corollary \ref{mu-symmetry-1} does not rely on Fact \ref{bv-sat-fact}. However beyond Fact \ref{bv-sat-fact}, the arguments of \cite{bv-sat-v3} (in particular the use of averages) have other applications (see for example the proof of solvability in \cite[Theorem 4.9]{gv-superstability-v5-toappear}). 

We can also say more on another result of Boney and the second author: \cite[Lemma 6.9.(2)]{bv-sat-v3} implies that, assuming $\mu$-superstability and $\mu$-tameness, there is a $\lambda_0 \ge \mu$ such that if $\seq{M_i : i < \delta}$ is a chain of $\lambda_0$-saturated models where $\delta \ge \lambda_0$ and $M_{i + 1}$ is universal over $M_i$, then $\bigcup_{i < \delta} M_i$ is saturated. We can improve this too:

\begin{corollary}
 Assume $\K$ is $\mu$-superstable and $\mu$-tame. Let $\delta$ be a limit ordinal and $\seq{M_i : i < \delta}$ is increasing in $\K_{\ge \mu}$ and $M_{i + 1}$ is universal over $M_i$ for all $i < \delta$. Let $M_\delta := \bigcup_{i < \delta} M_i$. If $\|M_\delta\| > \LS (\K)$, then $M_\delta$ is saturated.
\end{corollary} 
\begin{proof}
  By Proposition \ref{limit model lemma}, $M_\delta$ is a $(\lambda, \cf(\delta))$-limit model, where $\lambda = \|M_\delta\|$. By Fact \ref{ss-transfer}, $\K$ is $\lambda$-superstable. By Corollary \ref{mu-symmetry-1}, $\K$ has $\lambda$-symmetry. By Fact \ref{uq-limit}, $M_\delta$ is saturated.
\end{proof}

One can ask whether Corollary \ref{chain-sat-1} can be improved further by also getting the conclusion for $\lambda = \mu$. If $\mu = \LS (\K)$, it is not clear that $\LS (K)$-saturated models are the right notion so perhaps the right question is to be framed in terms of a superlimit. Recall from \cite[Definition N.2.2.4]{shelahaecbook} that a superlimit model is a universal model $M$ with a proper extension so that if $\seq{M_i : i < \delta}$ is an increasing chain with $M \cong M_i$ for all $i < \delta$, then (if $\delta < \|M\|^+$), $M \cong \bigcup_{i < \delta} M_i$. Note that, assuming $\mu$-superstability and uniqueness of limit models of size $\mu$, it is easy to see that the existence of a superlimit of size $\mu$ is equivalent to the statement that the union of an increasing chain of limit models in $\mu$ (of length less than $\mu^+$) is limit.

\begin{question}
  Assume $\K$ is $\mu$-tame and there is a type-full good $\mu$-frame on $\K_\mu$ (or just that $\K$ is $\mu$-superstable). Is there a superlimit model of size $\mu$?
\end{question}

We now turn to good frames and show that, assuming tameness, the statement of Theorem \ref{good-frame-weak-tameness} can be simplified. Recall that previous work of the second author gives a condition under which good frames can be constructed from tameness:

\begin{fact}[Theorem 10.8 in \cite{indep-aec-apal}]\label{good-frame-construct}
  Assume $\K$ is $\mu$-superstable and $\mu$-tame. If for any $\delta < \mu^+$, any chain of length $\delta$ of saturated models in $\K_{\mu^+}$ has a saturated union, then there is a type-full good $\mu^+$-frame with underlying class $\Ksatp{\mu^+}_{\mu^+}$.
\end{fact}

Combining this with Fact \ref{bv-sat-fact} it was proven in \cite{bv-sat-v3} that $\mu$-superstability and $\mu$-tameness implies the existence of a good $\lambda$-frame on the saturated models of size $\lambda$, for some high-enough $\lambda > \mu$. Now we show that we can take $\lambda = \mu^+$. We again give two proofs: one uses Theorem \ref{good-frame-weak-tameness} and the other relies on Fact \ref{good-frame-construct}.

\begin{corollary}\label{chain-sat-2}
  If $\K$ is $\mu$-superstable and $\mu$-tame, then there is a type-full good $\mu^+$-frame with underlying class $\Ksatp{\mu^+}_{\mu^+}$.
\end{corollary}
\begin{proof}[First proof]
  Combine Fact \ref{good-frame-construct} and Corollary \ref{chain-sat-1}.
\end{proof}
\begin{proof}[Second proof]
  By Fact \ref{ss-transfer}, $\K$ is superstable in every $\mu' \ge \mu$. Now apply Corollary \ref{good-frame-weak-tameness-2} (with $\lambda$ there standing for $\mu^+$ here).
\end{proof}
\begin{remark}\label{chain-sat-2-rmk}
  To obtain a type-full good $\lambda$-frame for $\lambda > \mu^+$, we can either make a slight change to the second proof of Corollary \ref{chain-sat-2}, or use the upward frame transfer of Boney and the second author \cite{ext-frame-jml, tame-frames-revisited-v6-toappear}.
\end{remark}

\section{Symmetry and categoricity}\label{sym-categ-sec}

Theorem \ref{transfer symmetry} has several applications to categorical AECs. We will use the following result, an adaptation of an argument of Shelah and Villaveces \cite[Theorem 2.2.1]{shvi635}, to settings with amalgamation:

\begin{fact}[The Shelah-Villaveces theorem, see in \cite{shvi-notes-v3-toappear}]\label{shelah-villaveces}
  Assume that $\K$ has no maximal models. Let $\mu \ge \LS (K)$. If $\K$ is categorical in a $\lambda > \mu$, then $\K$ is $\mu$-superstable.
\end{fact}

\begin{corollary}\label{categoricity symmetry corollary}
  Assume that $\K$ has no maximal models. Suppose $\lambda$ and $\mu$ are cardinals so that $\lambda>\mu\geq\LS(\K)$ and assume that $\K$ is categorical in $\lambda$. Then $\K$ is $\mu$-superstable and it has $\mu$-symmetry if at least one of the following conditions hold:

  \begin{enumerate}
  \item The model of size $\lambda$ is $\mu^+$-saturated.
  \item $\lambda \ge \hanf{\mu}$.
  \end{enumerate}
\end{corollary}
\begin{proof}
  By Fact \ref{shelah-villaveces}, $\K$ is $\chi$-superstable in every $\chi \in [\mu, \lambda)$. Now:

    \begin{enumerate}
    \item If the model of size $\lambda$ is $\mu^+$-saturated, then by Theorem \ref{categ-sym}, $\K$ has $\mu$-symmetry.
    \item If $\lambda \ge \hanf{\mu}$, then by Theorem \ref{sym-from-superstab}, $\K$ has $\mu$-symmetry.
    \end{enumerate}
\end{proof}

As announced in the introduction, we can combine Corollary \ref{categoricity symmetry corollary} with Fact \ref{uq-limit} to improve on \cite[Theorem 6.5]{sh394}. The following result also improves on Corollary 18 of \cite{vandieren-chainsat-apal}, by removing the successor assumption in the categoricity cardinal and obtaining uniqueness of limit models in much smaller cardinalities as well.

\begin{corollary}\label{categoricity uniqueness corollary}
  Assume that $\K$ has no maximal models. Suppose $\lambda$ and $\mu$ are cardinals so that $\lambda>\mu\geq\LS(\K)$ and assume that $\K$ is categorical in $\lambda$. If either $\cf(\lambda) > \mu$ or $\lambda \ge \hanf{\mu}$, then $\K$ has uniqueness of limit models of cardinality $\mu$.  That is, if $M_0,M_1,M_2\in\K_\mu$ are such that both $M_1$ and $M_2$ are limit models over $M_0$, then $M_1\cong_{M_0}M_2$.
\end{corollary}
\begin{proof}
Categoricity in $\lambda$, the assumption that $\cf(\lambda)>\mu$, and Fact \ref{shelah-villaveces} imply  that the model of cardinality $\lambda$ is $\mu^+$-saturated. We can now apply Corollary \ref{categoricity symmetry corollary}, to get that $\K$ is $\mu$-superstable and has $\mu$-symmetry.  Then Fact \ref{uq-limit} finishes the proof.
\end{proof}

Once we have obtained symmetry from a high-enough categoricity cardinal, we can deduce that the model in the categoricity cardinal has some saturation:

\begin{corollary}\label{cor-mu-sat}
  Let $\mu > \LS (\K)$. Assume that $\K$ is categorical in a $\lambda \ge \sup_{\mu_0 < \mu} \hanf{\mu_0^+}$. Then the model of size $\lambda$ is $\mu$-saturated. 
\end{corollary}
\begin{proof}
  By Fact \ref{jep-decomp}, we can assume without loss of generality that $\K$ has no maximal models. We check that the model of size $\lambda$ is $\mu_0^+$-saturated for every $\mu_0 \in [\LS (\K), \mu)$. Fix such a $\mu_0$. By Corollary \ref{categoricity symmetry corollary}, $\K$ is $\mu_0$-superstable, $\mu_0^+$-superstable, and has $\mu_0^+$-symmetry. By Fact \ref{union-sat-monica} $\Ksatp{\mu_0^+}$, the class of $\mu_0^+$-saturated models in $\K_{\ge \mu_0^+}$, is an AEC with Löwenheim-Skolem number $\mu_0^+$. Since it has arbitrarily large models, it must have a model of size $\lambda$, which is unique by categoricity.
\end{proof}

We conclude that categoricity in a high-enough cardinal implies some amount of weak tameness (a stronger result has been conjectured by Grossberg and the first author, see \cite[Conjecture 1.5]{tamenessthree}). We will use the notation from Chapter 14 of \cite{baldwinbook09}: we write $H_1$ for $\hanf{\LS (\K)}$ (see Definition \ref{hanf-def}) and $H_2$ for $\hanf{H_1} = \hanf{\hanf{\LS (\K)}}$.

\begin{corollary}\label{weak-tameness-from-categ}
  Let $\mu \ge \LS (\K)$. Let $\lambda \ge \hanf{\mu^+}$. If $\K$ is categorical in $\lambda$, then there exists $\chi < H_1$ such that $\K$ is $(\chi, \mu)$-weakly tame.
\end{corollary}
\begin{proof}
  By Corollary \ref{cor-mu-sat}, the model of size $\lambda$ is $\mu^+$-saturated. Now apply Fact \ref{weak-tameness-from-categ-fact}.
\end{proof}

We can derive a downward categoricity transfer. We will use the following fact, given by the proof of \cite[Theorem 14.9]{baldwinbook09} (originally \cite[II.1.6]{sh394}):

\begin{fact}\label{downward-categ-fact}
  If $\K$ is categorical in a $\lambda > H_2$, $\K$ is $(\chi, H_2)$-weakly tame for some $\chi < H_1$, and the model of size $\lambda$ is $\chi$-saturated, then $\K$ is categorical in $H_2$.
\end{fact}

\begin{corollary}\label{downward-categ}
  If $\K$ is categorical in a $\lambda \ge \hanf{H_2^+}$, then $\K$ is categorical in $H_2$.
\end{corollary}
\begin{proof}
  By Corollary \ref{weak-tameness-from-categ}, there exists $\chi < H_1$ such that $\K$ is $(\chi, H_2)$-weakly tame. By Corollary \ref{cor-mu-sat}, the model of size $\lambda$ is $\chi$-saturated. Now apply Fact \ref{downward-categ-fact}.
\end{proof}

We obtain in particular:

\begin{corollary}\label{fixed-point-downward}
  Let $\mu = \beth_\mu > \LS (\K)$. If $\K$ is categorical in some $\lambda > \mu$, then $\K$ is categorical in $\mu$.
\end{corollary}
\begin{proof}
  Without loss of generality (Fact \ref{jep-decomp}), $\K$ has no maximal models. Applying Corollary \ref{downward-categ} to $\K_{\ge \mu_0}$ for each $\mu_0 < \mu$, we get that $\K$ is categorical in unboundedly many $\mu_0 < \mu$. By (for example) Fact \ref{shelah-villaveces}, $\K$ is stable in every $\mu_0 < \mu$. Now let $M \in \K_\mu$. We show that $M$ is saturated, and this will imply categoricity in $\mu$. Let $N_0 \in \K_{<\mu}$ be such that $N_0 \lea M$. We want to show that every Galois type over $N_0$ is realized in $M$. Fix a categoricity cardinal $\mu_0 < \mu$ such that $\|N_0\| < \mu_0$. Let $N_0' \in \K_{\mu_0}$ be such that $N_0 \lea N_0' \lea M$. Then $N_0'$ is saturated (use stability in $\mu_0$ to build a $\theta^+$-saturated model of size $\mu_0$ for each $\theta < \mu_0$, then use categoricity in $\mu_0$). Thus $N_0'$ (and therefore $M$) realizes all Galois types over $N_0$, as desired.
\end{proof}

We can also build a good frame assuming categoricity in a high-enough cardinal (of arbitrary cofinality).

\begin{corollary}\label{good-frame-categ}
  Let $\mu \ge H_1$. Assume that $\K$ is categorical in a $\lambda > \mu$. If the model of size $\lambda$ is $\mu^+$-saturated (e.g.\ if $\cf (\lambda) > \mu$ or by Corollary \ref{cor-mu-sat} if $\lambda \ge \hanf{\mu^+}$), then there exists a type-full good $\mu$-frame with underlying class $\Ksatp{\mu}_{\mu}$.
\end{corollary}
\begin{proof}
  By Fact \ref{jep-decomp}, we can assume without loss of generality that $\K$ has no maximal models. By Fact \ref{weak-tameness-from-categ-fact}, there exists $\chi < H_1$ such that $\K$ is $(\chi, \mu)$-weakly tame. By Corollary \ref{categoricity symmetry corollary}, $\K$ has is $\chi'$-superstable and has $\chi'$-symmetry for every $\chi' \in [\chi, \mu]$. In particular, $\K$ has $\mu$-symmetry. Now apply Theorem \ref{good-frame-weak-tameness} with $(\mu, \lambda)$ there standing for $(\chi, \mu)$-here.
\end{proof}

The Hanf number of Corollary \ref{good-frame-categ} can be improved if we assume that the AEC is tame. We state a more general corollary summing up our results in tame categorical AECs:

\begin{corollary}\label{categ-tameness}
  Assume that $\K$ has no maximal models and is $\LS (\K)$-tame. If $\K$ is categorical in a $\lambda > \LS (\K)$, then:

  \begin{enumerate}
    \item For any $\mu \ge \LS (\K)$, $\K$ has uniqueness of limit models in $\mu$: if $M_0, M_1, M_2 \in \K_{\mu}$ are such that both $M_1$ and $M_2$ are limit over models $M_0$, then $M_1 \cong_{M_0} M_2$.
    \item For any $\mu > \LS (\K)$, $\Ksatp{\mu}$ is an AEC with $\LS (\Ksatp{\mu}) = \mu$ and there exists a type-full good $\mu$-frame with underlying class $\Ksatp{\mu}_{\mu}$.
  \end{enumerate}
\end{corollary}
\begin{proof}
  By Fact \ref{shelah-villaveces}, $\K$ is superstable in $\LS (\K)$. By Fact \ref{ss-transfer}, $\K$ is superstable in every $\mu \ge \LS (\K)$. By Corollary \ref{mu-symmetry-1}, $\K$ has symmetry in every $\mu \ge \LS (\K)$. The first part now follows from Fact \ref{uq-limit} and the second from Corollary \ref{chain-sat-1} and Corollary \ref{chain-sat-2} (together with Remark \ref{chain-sat-2-rmk}).
\end{proof}
\begin{remark}
  By \cite{tame-frames-revisited-v6-toappear}, we can transfer the type-full good $\LS (\K)^+$-frame on $\Ksatp{\LS (\K)^+}_{\LS (\K)^+}$ given by the previous corollary to a type-full good $(\ge \LS (\K)^+)$-frame with underlying class $\Ksatp{\LS (\K)^+}$. That is, the nonforking relation of the frame can be extended to types over all $\LS (\K)^+$-saturated models.
\end{remark}

\bibliographystyle{amsalpha}
\bibliography{transfer-sym-limit}

\end{document}